\documentclass[10pt,twoside]{IEEEtran}
\usepackage {amssymb,rotating,graphicx,cite, calc, color, psfrag}
\usepackage{authblk}

\usepackage{amsthm}
\theoremstyle{plain}

\newtheorem{theorem}{Theorem}{}
\newtheorem{proposition}{Proposition}{}
{}
{}
\newtheorem{defi}{Definition}{}
{}
\usepackage{epstopdf}
\usepackage[]{algorithm2e}
\usepackage {amssymb,rotating,graphicx,cite, calc, color, psfrag}
\usepackage{authblk}

\usepackage[cmex10]{amsmath}
\usepackage{amssymb,amsthm,amsfonts,mathrsfs}
\usepackage{graphicx}
\usepackage{enumitem,stackrel}
\usepackage{color}
\usepackage{mathtools,amssymb,bm}
\usepackage{amstext}
\usepackage{array}
\usepackage{algorithmicx}
\usepackage{graphicx}

\begin{document}

\title{Matrix Completion with Side Information using Manifold Optimization}


\author{Mohamad Mahdi~Mohades, Mohammad Hossein~Kahaei \thanks{The authors are with the School of Electrical Engineering, Iran University of Science \& Technology, Tehran 16846-13114, Iran (e-mail: mohamad\_mohaddes@elec.iust.ac.ir; kahaei@iust.ac.ir).}}%



\maketitle

\begin{abstract}

We solve the Matrix Completion (MC) problem based on manifold optimization by incorporating the side information under which the columns of the intended matrix are drawn from a union of low dimensional subspaces. It is proved that this side information leads us to construct new manifolds, as $\it{embedded}$ submanifold of the manifold of constant rank matrices, using which the MC problem is solved more accurately.
 The required geometrical properties of the aforementioned manifold are then presented for matrix completion.
Simulation results show that the  proposed method outperforms some recent techniques either based on side information or not.
\end{abstract}

{\bf{This paper is a preprint of a paper submitted to IET Signal Processing Journal.}}
\section{Introduction}\label{sec:Introduction}
Low-rank matrix completion (LRMC) can be somehow considered as a generalization of the recovery problem in compressed sensing (CS). While  LRMC benefits sparsity of the vector of singular values of a matrix,  CS enjoys sparsity in a given domain. However, one more challenging problem in the LRMC is the domain of sparsity of the singular values vector, which we may not be aware of. This, as a result,  prevents us from exploiting compressed sensing techniques in LRMC solutions. However, inspired from compressed sensing theorems, to complete the partially observed matrix ${\bf{M}} \in {\mathbb{R}^{m \times {n}}}$, the following optimization problem is minimized,

\begin {equation} \label {LRMC_Problem_Main}
\begin{array}{l}
\mathop {{\rm{minimize}}}\limits_{ {\bf{X}} \in {\mathbb{R}^{m \times {n}}}} \,\,\,\,\,\,\,rank{\left ( {\bf{X}} \right )}\,\,\,\,\,\,\,\,\,\,\,\\
\,\,\,\,\,\,\,\,\,\,\,\,\,{\rm{s}}{\rm{.t. \,\,\,\,}}{\rm{  }}{\mathrm{P}_\Omega}\left( {\bf{X}} \right) = {\mathrm{P}_\Omega} \left( {\bf{M}} \right)
\end{array},
\end{equation}
where $\Omega$ is a (random) subset of the set obtained by the Cartesian product of the sets $\{ 1,...,m\}$ and $ \{ 1,...,n\}$, and $\mathrm{P}_\Omega$ is the measurement operator which acts as the Hadamard product of a sampling matrix with $1$ and $0$ entries by the input matrix. Also, the nonzero entries of the aforementioned sampling matrix are given by $\Omega$ and $rank{\left (  \cdot \right )}$ represents the rank of a matrix.

When the rank of ${\bf{M}} \in {\mathbb{R}^{m \times {n}}}$   is known, say $r$, and there is uncertainty with measurements, the optimization problem is defined as
\begin {equation} \label {LRMC_Problem_Equi}
\begin{array}{l}
\mathop {{\rm{minimize}}}\limits_{ {\bf{X}} \in {\mathbb{R}^{m \times {n}}}} \,\,\,\,\,\,\,  \left\|{\mathrm{P}_\Omega}\left( {\bf{X}} \right) - {{\mathrm{P}_\Omega}\left( {\bf{M}} \right)} \right\|_{F}^2     \,\,\,\,\,\,\,\,\,\,\,\\
\,\,\,\,\,\,\,\,\,\,\,\,\,{\rm{s}}{\rm{.t. \,\,\,\,}}{\rm{  }}      rank{\left ( {\bf{X}} \right )}=r,
\end{array}
\end{equation}
where $\parallel \cdot \parallel_{F}$ shows the Frobenius norm.

Note that, due to the non-affine equality constraint, the above optimization problem is a non-convex one. Similarly, the problem (\ref{LRMC_Problem_Main}) is non-convex owing to non-convexity of the objective function. So far, a wide variety of techniques have been proposed to solve these problems. In a seminal work on the LRMC \cite {Candes_Recht}, ispiring from \cite{Fazel_Thesis}, authors apply a relaxation over the rank problem; which is similar to $l_1$-minimization addressed in the CS theorem \cite{candes2008introduction}. To do so, the nuclear norm; which is the $l_1$-norm of the singular values of the intended matrix, is utilized as
\begin {equation} \label {LRMC_Problem}
\begin{array}{l}
\mathop {{\rm{minimize}}}\limits_{ {\bf{X}} \in {\mathbb{R}^{m \times {n}}}} \,\,\,\,\,\,\,{\left \| {\bf{X}} \right \|_*}\,\,\,\,\,\,\,\,\,\,\,\\
\,\,\,\,\,\,\,\,\,\,\,\,\,{\rm{s}}{\rm{.t. \,\,\,\,}}{\rm{  }}{\mathrm{P}_\Omega}\left( {\bf{X}} \right) = {\mathrm{P}_\Omega} \left( {\bf{M}} \right)
\end{array},
\end{equation}
where $\parallel \cdot \parallel_{*}$ stands for the nuclear norm. The above formulation enjoys convexity for which convex optimization techniques can be applied. For example, by applying the Lagrangian method and making an unconstrained optimization problem, it is easy to use proximal algorithms  \cite{Combettes}, \cite {Candes_Tao_Power}.
 Moreover, the minimization problem (\ref{LRMC_Problem}) can be equivalently formulated as a semidefinite programming problem, as:
\begin {equation} \label {LRMC_Problem_Semidefinite}
\begin{gathered}
  \mathop {{\text{minimize}}}\limits_{{\mathbf{X}},{{\mathbf{W}}_1},{{\mathbf{W}}_2}} \,\,\,tr\left( {{{\mathbf{W}}_1}} \right) + tr\left( {{{\mathbf{W}}_2}} \right) \hfill \\
  \,\,\,\,\,\,\,\,{\text{s}}{\text{.}}\,\,{\text{t}}{\text{.}}\,\,\,\,\,\,\,\,\,\,{{\text{P}}_\Omega }\left( {\mathbf{X}} \right) = {{\text{P}}_\Omega }\left( {\mathbf{M}} \right) \hfill \\
  \,\,\,\,\,\,\,\,\,\,\,\,\,\,\,\,\,\,\,\,\,\,\,\,\,\left[ {\begin{array}{*{20}{c}}
  {{{\mathbf{W}}_1}}&{\mathbf{X}} \\ 
  {{{\mathbf{X}}^T}}&{{{\mathbf{W}}_2}} 
\end{array}} \right] \succcurlyeq {\mathbf{0}} \hfill \\ 
\end{gathered},
\end{equation}
where $tr(\cdot)$ is the trace of the input matrix, and ${\bf{Q}}\succcurlyeq \bf{0}$ states positive semidefiniteness of the matrix $\bf{Q}$ \cite {Candes_Recht}. Even though there are several software packages for solving semidefinite programs, they are mostly based on interior point methods which suffer from high computational complexity \cite{toh1999sdpt3}, \cite{dsdp5}. 
Another approach to solve the nondifferentiable convex Problem (\ref {LRMC_Problem}) is to use subdifferential concepts. In \cite{SVT}, authors have shown the following problem 
\begin{equation} \label{SVT_Pr}
\begin{gathered}
  \mathop {{\text{minimize}}}\limits_{\mathbf{X}} \,\,\,\tau {\left\| {\mathbf{X}} \right\|_*} + \frac{1}{2}\left\| {\mathbf{X}} \right\|_F^2 \hfill \\
  {\mkern 1mu} {\mkern 1mu} {\mkern 1mu} {\mkern 1mu} {\mkern 1mu} {\mkern 1mu} {\mkern 1mu} {\mkern 1mu} {\mkern 1mu} {\mkern 1mu} {\mkern 1mu} {\mkern 1mu} {\mkern 1mu} {\text{s}}.{\text{t}}.{\mkern 1mu} {\mkern 1mu} {\mkern 1mu} {\mkern 1mu} \,\,\,\,{{\text{P}}_\Omega }\left( {\mathbf{X}} \right) = {{\text{P}}_\Omega }\left( {\mathbf{M}} \right) \hfill \\ 
\end{gathered},
\end{equation}
for some $\tau>0$, is equivalent to Problem  (\ref {LRMC_Problem}). Then, by considering the Lagrangian dual problem of (\ref{SVT_Pr}),  they proposed an algorithm called singular value thresholding (SVT) which is based on proximal methods. Some literature have extended the noise free problem (\ref{SVT_Pr}) to the noisy case, where measurements are corrupted by noise \cite{SVT_noisy}.
There are some other approaches for minimizing nuclear norm minimization, such as robust principal component analysis (PCA)\cite{PCA}. In \cite{PCA} a variant of robust PCA has been introduced for solving matrix completion problem. Unfortunately, the method of \cite{PCA} is not robust against impulsive noise; for more information see \cite{li2019survey}. It is discussed that Frobenius norm cost function cannot deal with impulsive noise for the problem of matrix completion \cite{li2019survey}. Instead, some literature have utilized $l_p$-norm  for defining cost function of the matrix completion problem \cite{Outlier_lp} and \cite{Hap1}.

Apart from convex relaxation methods, it is possible to adopt non-convex optimization approaches to solve LRMC problems.  In \cite{LMAFIT}, by considering the low rank factorization ${{\bf{X}}_{m \times n}} = {{\bf{L}}_{m \times r}}{\bf{R}}_{n \times r}^T$, the following nonconvex optimization problem has been proposed for solving matrix completion problem,
\begin{equation}\label{Lmafit}
\begin{array}{*{20}{l}}
  {\mathop {{\text{minimize}}}\limits_{{\mathbf{X}},{\mathbf{L}},{\mathbf{R}}} {\mkern 1mu} {\mkern 1mu} {\mkern 1mu} {\mkern 1mu} {\mkern 1mu} {\mkern 1mu} {\mkern 1mu} \left\| {{\mathbf{L}}{{\mathbf{R}}^T} - {\mathbf{X}}} \right\|_F^2{\mkern 1mu} {\mkern 1mu} {\mkern 1mu} {\mkern 1mu} {\mkern 1mu} {\mkern 1mu} {\mkern 1mu} {\mkern 1mu} {\mkern 1mu} {\mkern 1mu} } \\ 
  {{\mkern 1mu} {\mkern 1mu} {\mkern 1mu} {\mkern 1mu} {\mkern 1mu} {\mkern 1mu} {\mkern 1mu} {\mkern 1mu} {\mkern 1mu} {\mkern 1mu} {\mkern 1mu} {\mkern 1mu} {\mkern 1mu} {\text{s}}.{\text{t}}.{\mkern 1mu} {\mkern 1mu} {\mkern 1mu} {\mkern 1mu}\,\,\,\,\,\,\, {{\text{P}}_\Omega }\left( {\mathbf{X}} \right) = {{\text{P}}_\Omega }\left( {\mathbf{M}} \right)} 
\end{array}.
\end{equation}
Then an algorithm has been proposed for solving Problem (\ref{Lmafit}) based on alternating minimization approach, in which  each of variables is alternatively fixed and the minimization problem is solved over the other variables. However, no global optimality guarantee has been proved therein.

With similar approach to \cite{LMAFIT}, Jain et al. \cite{Prateek}  have proposed the following non-convex optimization problem to solve matrix completion problem:
\begin{equation} \label{LR_Factor}
\mathop {{\rm{minimize}}}\limits_{{\bf{L}},{\bf{R}}} \left\| {{{\rm{P}}_\Omega }\left({\bf{M}}\right) - {{\rm{P}}_\Omega }\left( {{\bf{L}}{{\bf{R}}^T}} \right)} \right\|_F^2.
\end{equation}
To solve this non-convex problem, alternating minimization approach is used  and interestingly global optimality guarantees are presented. It is shown that the aforementioned simple matrix factorization approaches result in a better matrix completion performance rather than the nuclear norm minimization approach \cite {Prateek}.
Although, in  (\ref{Lmafit}) and (\ref{LR_Factor}), the rank of matrix is assumed known a priori, this problem may be solved for different rank values from 1 to a desirable one so that the correct value is found. A similar optimization problem is proposed in \cite {Burer} except that the objective function contains the summation $\left\|{\bf{L}}\right\|_{F}^2 + \left\|{\bf{R}}\right\|_{F}^2$ whose minimum value is equal to $\left\| {\bf{X}} \right\|_*$.
Some other matrix factorization based approaches for matrix completion problem can be found in \cite{yan2013exact} and \cite{he2011online}.

Another promising approach to solve non-convex optimization problems is based on smooth manifolds. Manifold structures have already been used for LRMC problems. In \cite {Keshavan_Few_Entries},   singular values of an incomplete matrix are  first approximated by applying the singular value decomposition (SVD) to a trimmed version of the partially observed matrix. Then, using manifold optimization on the Grassmann manifold, the corresponding singular vectors, and consequently, the completed matrix are extracted . In \cite{Dai_Milenkovic}, Dai et al. have shown how to utilize Grassmann manifold to find the row or column space which is consistent with the partial observations. They have also shown inappropriateness of the Frobenius norm used in (\ref{LRMC_Problem_Equi}) for their own formulation and accordingly suggested another metric for the LRMC. Finally, they have given some performance guarantees for some special cases of LRMC problems.  In \cite{RTRMC}, authors have defined a regularized objective function to recast the advantages of both \cite {Keshavan_Few_Entries} and \cite{Dai_Milenkovic}. The optimization problem they have defined, benefits a smooth objective function and a small search space compared to \cite {Keshavan_Few_Entries}. They have also utilized Riemannian trust region (RTR) method to solve the defined matrix completion problem.
Being aware of the fact that the set of all constant rank matrices offers a smooth manifold, Vandereycken has addressed the solution of (\ref{LRMC_Problem_Equi}) over such a manifold \cite {Bart}.
Mishra et al. have added a nuclear norm penalty to a convex function to solve the LRMC problem \cite {Mishra_Trace}. They have used a different representation of the manifold of constant rank matrices as the search space compared to that of \cite {Bart}. 

To solve the LRMC problem, no previously introduced approaches have incorporated the available side information of the matrix apart from the low rankness characteristic into computations. It has been shown that making use of such side information can reasonably enhance the results. In \cite {Xu_Jin_Zhou_2013}, it is  contemplated that the partially observed matrix can be written as ${\bf{M}}={{\bf{A}}}{{\bf{Z}}}{\bf{B}}^T$, where ${\bf{A}} \in {\mathbb{R}^{m \times {r_a}}}$ and ${\bf{B}} \in {\mathbb{R}^{n \times {r_b}}}$ are side information known, a priori, and ${\bf{Z}} \in {\mathbb{R}^{{r_a} \times {r_b}}}$ should be found. For instance, in the Netflix problem,  ${\bf{A}}$ and ${\bf{B}}$ are the available feature matrices of the users and movies, respectively, and ${\bf{Z}}$ is the unknown interaction matrix. It has been shown that by this means much fewer  number of revealed entries related to that of \cite{Candes_Recht} is required to solve the LRMC problem.
Similar use of side information has been addressed in \cite{Jain_Dhillon}.

To examine the impact of inexact side information, Chiang et al. have considered noisy observation along with the aforementioned factorization \cite{Chiang_Hsieh_Dhillon}. Specifically, they have assumed ${{\bf{M}}}={{\bf{A}}}{{\bf{Z}}}{\bf{B}}^T+{{\bf{R}}}$, where ${{\bf{R}}}$ corresponds to those parts of the low rank matrix ${{\bf{M}}}$ which cannot be modeled by the side information ${{\bf{A}}}$ and ${{\bf{B}}}$, and also {{\bf{Z}}} is a low rank matrix.
In \cite{Elhamifar_Mat_Completion}, authors have assumed that the columns of the intended matrix are drawn from a union of some low dimensional subspaces, and hence, have allowed high rank matrices to be considered. In other words, $\bf{M}$ benefits from the self expressive property $\bf{M}=\bf{MC}$, where $\bf{C}$ is a sparse matrix whose diagonal elements are zero. This property together with  that of \cite {Xu_Jin_Zhou_2013}, $i.e.$, ${\bf{M}}={{\bf{A}}}{{\bf{Z}}}{\bf{B}}^T$, has led the authors to successfully complete even high-rank matrices.
 
In the mentioned works, for the factorization ${\bf{M}}={{\bf{A}}}{{\bf{Z}}}{\bf{B}}^T$, it is assumed that $\bf{B}$ is accurate enough. However, in practice, this matrix which presents the related features of an application might be inaccurate. For instance, in the Netflix problem, $\bf{B}$ might represent the movie features with some uncertainties.\\Side information for matrix completion is not restricted to information regarding the matrix and the available side information can be about the sampling noise distribution \cite{Hap2}. However, this is not the concern of this work.

In this paper, we assume that  $\mathbf{M}$ can be factorized as ${\bf{M}}={{\bf{A}}}{{\bf{Z}}}{\bf{B}}^T$ and similar to \cite{Elhamifar_Mat_Completion}, the basis ${\bf{B}}$ lies in a union of subspaces, and an approximation of this basis, say ${\mathbf{B}'}$, is available. Also, as opposed to \cite{Elhamifar_Mat_Completion}, ${\bf{A}}$ can be unknown and estimation of ${\bf{B}}$ can be inaccurate.
Contemplating the mentioned assumptions, we formulate a non-convex matrix completion problem and give its straightforward iterative solution. Our solution comprises two steps at each iteration. In the first step,  having the expression matrix ${{\bf{C}}_{k-1}}$ at iteration $k$, we consider  a new submanifold and solve an optimizition problem over it to complete the  matrix ${{\bf{X}}_{k}}$. The optimal value is found by utilizing the gradient descent method.
At the second step of iteration $k$, we find a new expression matrix which better expresses the completed matrix ${{\bf{X}}_{k}}$. This is performed by minimizing the $l_1$-norm of the vectorized ${{\bf{C}}}$ under the constraints ${\bf{X}}_k={{\bf{X}}_{k}{\bf{C}}}$ and $diag({\bf{C}})=\bf{0}$.

The proposed approach is more reliable and general than the other side information-based MC problems. For instance, unlike \cite{Elhamifar_Mat_Completion}, we only incorporate an approximate side information basis ${\bf{B}}'$, not the exact basis ${\bf{B}}$. Meanwhile, due to using the manifold optimization approach, no rank minimization relaxation is considered. In addition, it can be shown that our method is computationally more effective than that of \cite{Elhamifar_Mat_Completion}.
The main differences with the existing manifold optimization approaches for matrix completion are as follows. Firstly, unlike the existing approaches, which utilize well-studied manifold structures with known geometry properties, we propose a new manifold by using the existing theorems. Secondly, by utilizing side information, our approach can deal with high rank regimes. Simulation results also confirm our theoretical discussions.

The rest of paper is organized as follows. In Section \ref{sec:Preliminary} some preliminaries are presented to introduce our work. Main results are presented in Section \ref{Main_results}. Section \ref{Simulation} is devoted to illustrate simulation results. Section \ref{Conclusion} concludes the paper.

\section{Preliminaries}\label{sec:Preliminary}
In this section, we give a brief review on the manifold of constant rank matrices and the manifold optimization technique we intend to use. Note that although there are different representations for such a manifold, we only present the one which is of our interest.
\subsection {Manifold of Constant Rank Matrices} \label{Manifold_CRM}
To introduce the manifold of constant rank matrices, we first present some definitions.

\begin {defi} \label {Topology}
A set $ {X}$ along with a collection of its subsets, $T$, is called a topology, or topological space, if the following properties are satisfied by the subsets in $T$,
\begin{enumerate}
  \item The trivial subsets $X$ and $\emptyset$ belong to $T$,
\item The intersection of a finite number of sets in $T$ belongs to $T$,
\item The union of an arbitrary number of sets in $T$ belongs to $T$.
\end{enumerate}

\end{defi}

\begin {defi} \label {Chart}
Let $\mathcal{M}$ be a set. A given subset $\mathcal{U}$ of $\mathcal{M}$ together with a bijective function $\varphi$ between  $\mathcal{U}$  and an open subset of  $\mathbb{R}^d$ consist a pair  $\left(\mathcal{U},\varphi \right)$ which is called a $d$-dimensional chart of the set  $\mathcal{M}$  \cite {Absil_Book}.
\end{defi}
The above definition lets us better study functions over a given set $\mathcal{M}$. However, to have meaningful results, it is required that the intersection of any two different charts be compatible. For example, consider two charts $\left(\mathcal{U}_1,\varphi_1 \right)$ and $\left(\mathcal{U}_2,\varphi_2 \right)$ over the set  $\mathcal{M}$ and consider the real valued function $f$ defined on  $\mathcal{U}_1 \cap \mathcal{U}_2$. Then, $f \circ \varphi _1^{ - 1}$ and $f \circ \varphi _2^{ - 1}$ should have the same differentiability properties on $\mathcal{U}_1 \cap \mathcal{U}_2$. Taking into account such required properties, we present the definition of an atlas over the set $\mathcal{M}$.

\begin {defi}  \label{Atlas}
A collection of charts $\left(\mathcal{U}_\alpha,\varphi_\alpha \right)$ of the set $\mathcal{M}$ with the next properties consists an atlas,
\begin{itemize}
  \item $\cup_\alpha\mathcal{U}_\alpha=\mathcal{M}$,
  \item The elements of the atlas smoothly overlap. In other words, for any pair $\alpha, \beta $ with $\mathcal{U}_\alpha \cap \mathcal{U}_\beta \neq \emptyset$, the sets $\varphi_\alpha\left(\mathcal{U}_\alpha \cap \mathcal{U}_\beta\right)$ and $\varphi_\beta\left(\mathcal{U}_\alpha \cap \mathcal{U}_\beta\right)$ are open sets in $\mathbb{R}^d$ and the coordinate change ${\varphi _\beta } \circ \varphi _\alpha ^{ - 1}:{\mathbb{R}^d} \to {\mathbb{R}^d}$ is smooth on its domain, $i.e.$, infinitely differentiable \cite {Absil_Book}.
\end{itemize}
\end{defi}

\begin {defi}  \label{Atlas_Maximal}
Let $\mathcal{A}$ be an atlas on the set $\mathcal{M} $.  Now, assume that $\mathcal{A}^+ $ is the set of all pairs $\left(\mathcal{U},\varphi \right)$ so that $\mathcal{A} \cup\left(\mathcal{U},\varphi \right) $ is still an atlas. Then, $\mathcal{A}^+ $ is called the maximal atlas or complete atlas generated by $\mathcal{A}$. This maximal atlas is called a differntiable structure on $\mathcal{M}$ \cite {Absil_Book}.
\end{defi}

From the above definition, we see that the maximal atlas is offering a topology. In special cases of this topology, manifolds are defined as follows.

\begin {defi} \label{Manifold}
The couple  $\left(\mathcal{M}, \mathcal{A}^+ \right)$ is a $d$-dimensional manifold provided that the topology induced by $\mathcal{A}^+ $ is Hausdorff and second countable. Note that the dimension is the same as the dimension of charts as expressed in Definition \ref{Chart}.
A Hausdorff space is a topological space in which any two distinct elements can be separated by disjoint open sets. Moreover, a second countable space is a topological space whose topological basis is countable  \cite {Absil_Book}.
\end{defi}

From the above definitions, it might be concluded that to construct a manifold, we must find a maximal atlas with the properties given in  Definition \ref{Manifold}. However, there are many other theorems applied for construction of manifolds using some easy steps. As such, a critical theorem to our contributions is presented in the following to generate new manifolds from the existing ones.
First, the definition of embedded submanifolds is given.

\begin {defi}  \label{Submanifold}
Let $\left(\mathcal{N}, \mathcal{B}^+ \right)$ and $\left(\mathcal{M}, \mathcal{A}^+ \right)$ be two manifolds and $\mathcal{N}\subset \mathcal{M}$. Then,  $\left(\mathcal{N}, \mathcal{B}^+ \right)$ is called an embedded submanifold of manifold  $\left(\mathcal{M}, \mathcal{A}^+ \right)$ provided that the topology of  $\left(\mathcal{N}, \mathcal{B}^+ \right)$ coincides with the topology induced by manifold  $\left(\mathcal{M}, \mathcal{A}^+ \right)$.
Moreover, if  $\mathcal{N}\subset \mathcal{M}$, then, there exists at most one differentiable structure that makes  $\mathcal{N}$ an embedded submanifold  \cite {Absil_Book}.
\end{defi}
Embedded submanifolds are of great interest due to inheriting the topological properties of the manifold they are obtained from.

\begin {defi}  \label{Rank}
Let $F:\mathcal{M}_1\to\mathcal{M}_2$ be a functoin from manifold $\mathcal{M}_1$ to manifold $\mathcal{M}_2$ whose dimensions are $d_1$ and $d_2$ , respectively. Then, the coordinate representation of $F$ around $x\in\mathcal{M}$ can be defined by considering chart $\varphi_1$ around $x$  and chart $\varphi_2$ around $F(x)$ as follows:
\begin{equation} \label {coordinate_representation}
\hat F = {\varphi _2} \circ F \circ \varphi _1^{ - 1}:{\mathbb{R}^{{d_1}}} \to {\mathbb{R}^{{d_2}}}.
\end{equation}
Moreover, $F$ is differentiable or smooth at $x$ if $\hat F$ is differentiable at $\varphi_1(x)$.
Also, the rank of $F$ at point $x\in\mathcal{M}$ is defined as the dimension of the range of  the differential of $\hat F$ at $\varphi_1(x)$ \cite{Absil_Book}.
\end{defi}
Now, we present a critical theorem used in our work.

\begin {theorem} \label{Submersion_Theorem}
Let $F:\mathcal{M}_1\to\mathcal{M}_2$ be a smooth function between two manifolds of dimension $d_1$ and $d_2$, and consider $y$ as a point of $F(\mathcal M)$.
If $F$ has a constant rank $k<d_1$ in a neighborhood of $F^{-1}(y)$, then $F^{-1}(y)$ is a closed embedded submanifold of $\mathcal{M}_1$ of dimension  $d_1-k$  \cite{Absil_Book}.
\end{theorem} \
To utilize manifolds in optimization problems, we need to know about some differential concepts of manifolds, such as the tangent space. In the following, we review some of the required differential geometry concepts of the manifold of constant rank matrices.

\begin{theorem}  \label{manifold_CRM}
The set $\mathcal{M}^{(r)}$, consisting of all $m\times n$ real valued matrices of rank $r$, is a smooth submanifold of dimension $(m+n-r)r$ embedded in $\mathbb{R}^{m\times n}$. Moreover, let the SVD of the point ${\bf{X}}\in\mathcal{M}^{(r)}$ be  ${\bf{X}}={\bf{U\Sigma V}}^T$. Then, the tangent space of $\mathcal{M}^{(r)}$ at ${\bf{X}}$, denoted by $T_{\bf{X}}\mathcal{M}^{(r)}$, is defined as  \cite{Bart}
\begin{equation} \label{tangent_space_CRM}
\resizebox{.99 \hsize}{!}{$
\begin{gathered}
  {T_{\mathbf{X}}}{\mathcal{M}^{(r)}} = \left\{ {\left[ {{\mathbf{U}}{\mkern 1mu} {\mkern 1mu} {{\mathbf{U}}_ \bot }} \right]\left[ {\begin{array}{*{20}{c}}
  {{\mathbb{R}^{r \times r}}}&{{\mathbb{R}^{r \times \left( {n - r} \right)}}} \\ 
  {{\mathbb{R}^{\left( {m - r} \right) \times r}}}&{{{\mathbf{0}}_{\left( {m - r} \right) \times \left( {n - r} \right)}}} 
\end{array}} \right]{{\left[ {{\mathbf{V}}{\mkern 1mu} {\mkern 1mu} {{\mathbf{V}}_ \bot }} \right]}^T}} \right\} \hfill \\
  \,\,\,\,\,\,\,\, = {\mkern 1mu} {\mkern 1mu} {\mkern 1mu} {\mkern 1mu} {\mkern 1mu} \left\{ {{\mathbf{UM}}{{\mathbf{V}}^T} + {{\mathbf{U}}_p}{{\mathbf{V}}^T} + {\mathbf{UV}}_p^T:} \right.{\mathbf{M}} \in {\mathbb{R}^{r \times r}}, \hfill \\
  \left. {{\mkern 1mu} {\mkern 1mu} {\mkern 1mu} {\mkern 1mu} {\mkern 1mu} {\mkern 1mu} {\mkern 1mu} {\mkern 1mu} {\mkern 1mu} {\mkern 1mu} {\mkern 1mu} {\mkern 1mu} {\mkern 1mu} {\mkern 1mu} {\mkern 1mu} {\mkern 1mu} {\mkern 1mu} {\mkern 1mu} {\mkern 1mu} {\mkern 1mu} {\mkern 1mu} {\mkern 1mu} {\mkern 1mu} {\mkern 1mu} {{\mathbf{U}}_p} \in {\mathbb{R}^{m \times r}},{\mathbf{U}}_p^T{\mathbf{U}} = {\mathbf{0}},{{\mathbf{V}}_p} \in {\mathbb{R}^{n \times r}},{\mathbf{V}}_p^T{\mathbf{V}} = {\mathbf{0}}} \right\} \hfill \\ 
\end{gathered}  $}.
\end{equation}
\end{theorem}

\begin {defi}
The disjoint union of all tangent spaces is called tangent bundle. In the case of manifold ${\mathcal M}^{(r)}$, the tangent bundle is written as   \cite{Absil_Book}
\begin{equation}
\begin{array}{l}
T{{\mathcal M}^{(r)}}: = \mathop  \cup \limits_{{\bf{X}} \in {{\mathcal M}^{(r)}}} \left\{ {\bf{X}} \right\} \times {T_{\bf{x}}}{{\mathcal M}^{(r)}} = \\
\,\,\,\,\,\,\,\,\,\,\,\left\{ {\left( {{\bf{X}},\xi } \right) \in {\mathbb{R}^{m \times n}} \times {\mathbb{R}^{m \times n}}:{\bf{X}} \in {{\mathcal M}^{(r)}},\xi  \in {T_{\bf{X}}}{{\mathcal M}^{(r)}}} \right\}
\end{array}.
\end{equation}
\end{defi}

To evaluate the convergence of any optimization algorithm, a metric should be considered. One rational and commonly used metric is introduced by the inner product. The inner product of the Euclidean space $\mathbb{R}^{m\times n}$ is defined as, $\left\langle {{\bf{A}},{\bf{B}}} \right\rangle  = {\rm{tr}}\left( {{{\bf{A}}^T}{\bf{B}}} \right)$, where  ${\bf{A}},{\bf{B}}\in\mathbb{R}^{m\times n}$. Now, by restricting the mentioned inner product to the tangent bundle $T{{\mathcal M}^{(r)}}$, the manifold ${\mathcal M}^{(r)}$ turns into a Riemannian manifold. The deduced metric is denoted by either $g_{\bf{X}}\left ( {\xi , \eta} \right )$ or $\left<\xi , \eta\right>_{\bf{X}}$, whose entries are tangent vectors belonging to the tangent space at point $\bf{X}$.  In this case, the manifold is said to be endowed with the Riemannian metric $g$.

Now, we define the concept of gradient on a Riemannian manifold which is a crucial ingradient of descent optimization approaches.
\begin{defi} \label{gradient} 
Let $f$ be a scalar valued function over the Riemannian manifold $\mathcal{M}^{(r)}$ endowed with the metric $g$. Then, the gradient of $f$ at point ${\bf{X}}\in\mathcal{M}^{(r)}$ denoted by ${\rm{grad}}f({\bf{X}})$ is the unique element of the tangent space  $T_{\bf{X}}\mathcal{M}^{(r)}$ which satisfies the equality,
\begin {equation}
g_{\bf{X}}\left({\rm{grad}}f({\bf{X}}),\xi\right)={\rm{D}}f({\bf{X}})[\xi] \,\,\,\,\,\,  \forall \xi \in T_{\bf{X}}\mathcal{M}^{(r)},
\end {equation}
where ${\rm{D}}f({\bf{X}})[\xi]$ denotes the directional derivative of $f$ along $\xi$ \cite{Absil_Book}.
\end{defi}

After finding the gradient of a function over a manifold, it is easy to derive the gradient of the mentioned function over any embedded submanifold as next proposition states.
\begin {proposition} \label{gradient_submanifold} 
Let $\bar{f}$ be a function over the Riemannian manifold $\bar{\mathcal{M}}$ whose gradient at point ${\bf{X}}\in\bar{\mathcal{M}}$ is shown by ${\rm{grad}}\bar{f}({\bf{X}})$, and $f$ be the restriction of $\bar{f}$ to the embedded submanifold $\mathcal{M}$. Then, the gradient of ${f}$  at point ${\bf{X}}\in{\mathcal{M}}$ is obtained as
\begin{equation}\label{Project_Gradient}
{\rm{grad}}{f}({\bf{X}}) = {\rm{P}}_{\bf{X}}{\rm{grad}}\bar{f}({\bf{X}}),
\end{equation}
where ${\rm{P}}_{\bf{X}}$ shows the orthogonal projection onto $T_{\bf{X}}\mathcal{M}$  \cite{Absil_Book}.
\end{proposition}

By applying optimization methods, we would find points on the tangent spaces which may not lie on the desired manifold. To map the obtained points to the manifold, the retraction function defined in the following is utilized.
\begin{defi}\label{retraction_defi}
Let $R:T\mathcal{M}\to\mathcal{M}$ be a smooth mapping and $R_{\bf{X}}$ be the restriction of $R$ to the tangent space $T_{\bf{X}}\mathcal{M}$. Then, $R$ is a retraction on the manifold $\mathcal{M}$, provided that,\\
$(1)$ $R_{\bf{X}}({\bf{0}}_{\bf{X}}) = {\bf{X}}$, where ${\bf{0}}_{\bf{X}}$ is the zero element of $T_{\bf{X}}\mathcal{M}$,
\\$(2)$ ${\rm{D}}R_{\bf{X}}({\bf{0}}_{\bf{X}}) = {\rm{id}}_{T_{\bf{X}}\mathcal{M}}$, where ${\rm{id}}_{T_{\bf{X}}\mathcal{M}}$ is the identity mapping on ${T_{\bf{X}}\mathcal{M}}$ \cite{Absil_Book}.
\end{defi}
One way to derive this retraction function is to use the metric projection defined as follows.

\begin{defi} \label {retraction} 
Let $\mathcal{M}$ be a manifold. Then, the metric projection retraction function $R_{\bf{X}}:T_{\bf{X}}\mathcal{M}\to\mathcal{M}$; which maps ${\bf{X}+\xi}$ to the element ${\bf{Z}}\in\mathcal{M}$ for a given element $\xi\in T_{\bf{X}}\mathcal{M}$, is defined as \cite{Absil_Paper}
\begin{equation}
R_{\bf{X}}(\xi)=\mathop{\arg \min}\limits_{{\bf{Y}}\in\mathcal{M}}\left\|{\bf{X}+\xi}-{\bf{Y}}\right\|_F.
\end{equation}
\end{defi}

Considering Definition \ref{retraction}, the retraction on the manifold of constant rank matrices $\mathcal{M}^{(r)}$ is to make all singular values equal to zero except the $r$ largest ones. Note that when the number of nonzero singular values is less than $r$, there is no result for retraction and when the $r^{th}$ and $(r+1)^{th}$ singular values are equal, retraction is not unique. These can be ignored due to the fact that  retraction can be considered locally and not necessarily all over the tangent bundle.

\section {Main Results} \label{Main_results}
To exploit self-expressivity for matrix completion, the intended matrix is supposed to be factorized as ${{\bf{M}}}={{\bf{A}}}{{\bf{Z}}}{\bf{B}}^T$, where
${\bf{A}} \in {\mathbb{R}^{m \times {r}}}$, ${\bf{B}} \in {\mathbb{R}^{n \times {r}}}$ and ${\bf{Z}} \in {\mathbb{R}^{{r} \times {r}}}$ is unknown and full rank. We moreover assume that an inaccurate estimation of the  basis ${\bf{B}} \in {\mathbb{R}^{n \times {r}}}$, denoted by  ${\bf{B}'}$, is available. Without loss of generality, we consider $r\leq m \leq n$, where $r $ is the rank of $\bf{M}$. Then, we define the following optimization problem:
\\
\begin {subequations} \label {MC_Problem_Ours_1}
\begin{align}\label {MC_Problem_Ours_1_a}
 \mathop {{\rm{minimize}}}\limits_{{\bf{X}},{\bf{C}}} \,\,\,\,\,\,\,  \frac{1}{2}\left\|{\mathrm{P}_\Omega}\left( {\bf{X}} \right) - {{\mathrm{P}_\Omega}\left( {\bf{M}} \right)}\right\|_F^2+\lambda \left\|\bf{C}\right\|_0   \,\,\,\,\,\,\,\,\,\,\,\\ 
\label {MC_Problem_Ours_1_b}{\rm{s}}{\rm{.t. \,\,\,\,}}{\rm{  }}  \,\,\,    rank{\left ( {\bf{X}} \right )}=r \,\,\,\,\,\,\,\,\,\,\,\,\,\,\,\,\,\,\,\,\,\,\,\,\,\,\,\,\,\,\,\,\,\,\,\,\,\,\,\,\,\,\,\,\\
\label{MC_Problem_Ours_1_c} {\bf{XC} } = {\bf{X}}\,\,\,\,\,\,\,\,\,\,\,\,\,\,\,\,\,\,\,\,\,\,\,\,\,\,\,\,\,\,\,\,\,\,\,\,\,\,\,\,\,\,\,\,\,\,\,\,\,\,\,\,\,\,\,\,\\
\label {MC_Problem_Ours_1_d}
 {diag(\bf{C}) } = {\bf{0}},\,\,\,\,\,\,\,\,\,\,\,\,\,\,\,\,\,\,\,\,\,\,\,\,\,\,\,\,\,\,\,\,\,\,\,\,\,\,\,\,\,\,\,\,
\end{align}
\end{subequations}
where ${\|\bf{C}\|}_0$ shows the number of non-zero elements of the matrix $\bf{C}$, and along with conditions (\ref{MC_Problem_Ours_1_c}) and (\ref{MC_Problem_Ours_1_d}) promote the self expressive property of the desired matrix. Also, $diag(\cdot)$ is the diagonal of a matrix and $\lambda$ shows the regularization parameter to compromise between the self expressive property and the completion error. The rationale behind minimizing $\left\|\bf{C}\right\|_0$ is that each column of the self-expressive matrix $\bf{X}$ ought to be written as a linear combination of as few as possible of the other columns.
This optimization problem is non-convex in the Euclidean space because of noncovexity of both objective function and the equality constraints (\ref{MC_Problem_Ours_1_b}) and (\ref{MC_Problem_Ours_1_c}).  We saw in Section \ref{sec:Preliminary} that the equality constraint (\ref{MC_Problem_Ours_1_b}) is addressing the manifold of constant rank matrices, $\mathcal{M}^{(r)}$. As a result, we remove the equality constraint (\ref{MC_Problem_Ours_1_b}) at this stage, and rewrite (\ref{MC_Problem_Ours_1}) over the manifold $\mathcal{M}^{(r)}$ as
\begin {equation} \label {MC_Problem_Ours_2}
\begin{array}{l}
 \,\,\,  \mathop {{\rm{minimize}}}\limits_{{\bf{X}}\in\mathcal{M}^{(r)},{\bf{C}}} \,\,\,\,\,\,\,  \frac{1}{2}\left\|{\mathrm{P}_\Omega}\left( {\bf{X}} \right) - {{\mathrm{P}_\Omega}\left( {\bf{M}} \right)}\right\|_F^2+\lambda \left\|\bf{C}\right\|_0     \,\,\,\,\,\,\,\,\,\,\,\\
 \,\,\,\,\,\,\,\,\,\,\,\,\,\,\,{\rm{s}}{\rm{.t. \,\,\,\,}} \,\,\,\,\,\,\,\,\,\,\,\,\,\,\,{\rm{  }}\,\,\,\,\, {\bf{XC} } = {\bf{X}}\\
 \,\,\,\,\,\,\,\,\,\, \,\,\,\,\,\,\,\,\,\,\,\,\,\,\,\,\,\,\,\,\,\,\,\,\,\,\,\,\,\,\,\,\,\,\, {diag(\bf{C}) } = {\bf{0}}.\\
\end{array}
\end{equation}
Now, we propose an alternating minimization approach to solve (\ref{MC_Problem_Ours_2}) in two steps by fixing each variable in each step and minimizing over the other variable as follows:
\begin {subequations} \label {Alternating_Minimization}
\begin{align}\label {Alternating_Minimization_a}
\begin{array}{l}
{{\bf{X}}_{k + 1}} = \mathop {{\rm{arg min}}}\limits_{{\bf{X}} \in \mathcal{M}^{(r)}} \,\,\left\| {{{\rm{P}}_\Omega }\left( {\bf{X}} \right) - {{\rm{P}}_\Omega }\left( {\bf{M}} \right)} \right\|_F^2\\
\,\,\,\,\,\,\,\,\,\,\,\,\,\,\,\,\,\,\,\,\,\, \,\,\,\,\,\,\,{\rm{s.t.}}\,\,\,\,\,\,\,\,\,\,\,{\bf{X}}{{\bf{C}}_{k}} = {\bf{X}}\\
\end{array} \,\,\,\,\,\,\,\,\,\,\,\,\,\,\, \,\,\,\,\,\,\,\,\,\,\,\,\,\,\,\\
\begin{array}{l}
{{\bf{C}}_{k + 1}} = \mathop {{\rm{arg min}}}\limits_{\bf{C}} \,\,\,\,\,\,{\left\| {\bf{C}} \right\|_0}\\
\,\,\,\,\,\,\,\,\,\,\,\,\,\,\,\, \,\,\,\,\,\,\,\,\,\,\,\,\,\,\,{\rm{s.t.}}\,\,\,\,\,\,\,\,\,\,\,{{\bf{X}}_{k+1}}{\bf{C}} = {{\bf{X}}_{k+1}}\\
\,\,\,\,\,\,\,\,\,\,\,\,\,\,\,\,\,\,\,\,\,\,\,\,\,\,\,\,\,\,\,\,\,\,\,\,\, \,\,\,\,\,\,\,\,\,\,\,\,diag\left( {\bf{C}} \right) = {\bf{0}},\\
\end{array} \,\,\,\,\,\,\,\,\,\,\,\,\,\,\, \,\,\,\,\,\,\,\,\,\,\,\,\,\,\, \,\,\,\,\,\,\,\,\,\,
\end{align}
\end{subequations}
where the subscript $k+1$ stands for the $(k+1)^{th}$ iteration.
\\
\indent The first problem of (\ref{Alternating_Minimization}) is a constrained optimization problem over the manifold $\mathcal{M}^{(r)}$ for solving which we can use the following proposition to reduce it to an unconstrained problem over an embedded submanifold of $\mathcal{M}^{(r)}$.
The second minimization problem is a constrained non-convex problem whose non-convexity is due to the term ${\|\bf{C}\|}_0$. It is well-known that a convex surrogate for ${\|\bf{C}\|}_0$ is ${\|\bf{C}\|}_1$, which is the $l_1$-norm of a vector obtained by vectorizing $\bf{C}$.

\begin {proposition} \label{Proposition_Ours_1}
Let the set $\mathcal{M}^{(r)(\bf{C})}$ contain the points  ${\bf{X}}\in\mathcal{M}^{(r)}$ which satisfy the self-expressive property ${\bf{XC}}={\bf{X}}$ for a given ${\bf {C}} \in \mathbb{R}^{n\times n}$. Then, $\mathcal{M}^{(r)(\bf{C})}$ is an embedded submanifold of dimension $(m+n-r)r-q$, where $q$ is the rank of the  matrix ${\bf{C}}-{\bf{I}}$ and $\bf {I}$ is the identity matrix.
\end{proposition}
\begin {proof}
First, consider the function $F_{\bf{C}}:\mathcal{M}^{(r)}\to\mathbb{R}^{m\times n}$, so that for a given ${\bf{X}}\in\mathcal{M}^{(r)}$, we can write $F_{\bf{C}}({\bf{X}})={\bf{X}}({\bf{C}}-{\bf{I}})$.  This function is linear, and hence, a constant rank function whose rank is equal to the rank of the matrix ${\bf{C}}-{\bf{I}}$, say $q$.
Now, it is enough to apply Theorem \ref{Submersion_Theorem} for the point $\mathbf{0}\in\mathbb{R}^{m\times n}$. The point $\mathbf{0}\in\mathbb{R}^{m\times n}$ does belong to $F_{\bf{C}}(\mathcal{M}^{(r)})$ and $F_{\bf{C}}$ is a function of constant rank $q$ for any neighborhood of $F_{\bf{C}}^{-1}(\bf{0})$. Therefore by Theorem \ref{Submersion_Theorem},  $F_{\bf{C}}^{-1}(\bf{0})$, which is $\mathcal{M}^{(r)(\bf{C})}$, is a closed embedded submanifold of dimension $(m+n-r)r-q$ as stated above.
\end{proof}

Considering Proposition \ref{Proposition_Ours_1} and the mentioned $l_1$-norm relaxation,  (\ref{Alternating_Minimization}) can be rewritten as

\begin{subequations}\label{Alternating_Minimization_Surrogate}
\begin{align} \label{Alternating_Minimization_Surrogate_a}
{{\bf{X}}_{k + 1}} = \mathop {{\rm{arg min}}}\limits_{{\bf{X}} \in \mathcal{M}^{(r)(\bf{C}_{k})}} \,f({\bf{X}})= \left\| {{{\rm{P}}_\Omega }\left( {\bf{X}} \right) - {{\rm{P}}_\Omega }\left( {\bf{M}} \right)} \right\|_F^2\\
\begin{array}{l}{{\bf{C}}_{k + 1}} \,\,= \,\,\,\,\,\mathop {{\rm{arg min}}}\limits_{\bf{C}} {\left\| {\bf{C}} \right\|_1} \\
\,\,\,\,\,\,\,\,\,\,\,\,\,\,\,\,\,\,\,\,\,\,\,\,\,\,\,\,\,\,\,\,\,\,\,{\rm{s.t.}}\,\,\,\,\,\,\,{{\bf{X}}_{k + 1}}{\bf{C}} = {{\bf{X}}_{k + 1}}\\
\,\,\,\,\,\,\,\,\,\,\,\,\,\,\,\,\,\,\,\,\,\,\,\,\,\,\,\,\,\,\,\,\,\,\,\,\,\,\,\,\,\,\,\,\,\,\,\,\,\,\,\,diag\left( {\bf{C}} \right) = {\bf{0}}\label{Alternating_Minimization_Surrogate_b},
\end{array}\,\,\,\,\,\,\,\,\,\,\,\,\,\,\,\,\,\,\,\,\,\,\,\,\,\,\,\,\,\,\,\,\,\,\,\,\,\,
\end{align}
\end{subequations}

in which the second minimization problem is  convex and can be easily solved using diverse convex optimization approaches. Also, the first minimization is an unconstrained manifold optimization problem whose objective function is differentiable. To solve (\ref{Alternating_Minimization_Surrogate_a}), it is easier to use the gradient descent method on manifolds which does not require some geometrical properties of manifolds such as affine connection (See Algorithm 1). For iterative solution of (\ref{Alternating_Minimization_Surrogate}), an appropriate initial point, say ${\bf{C}}_0$, should be  selected  so that the final result lies as close as possible to the global point. For this purpose, we  utilize the available basis ${\bf{B}'}$ such that the matrix ${\bf{C}}_0$ to be selected as the sparsest matrix which satisfies ${\bf{B}}'^T({\bf{C}}_0-{\bf{I}})={\bf{0}}$ and $diag({\bf{C}}_0)={\bf{0}}$.

\small{
\begin{center} \label{Line_Search_Absil}
 \begin{tabular}{|c|}
  \hline
\hspace{-2.8cm}\textbf{Algorithm 1}: Gradient Descent Method\\ \hspace{-1.cm}on a Riemannian Manifold \\
 \hline
\hspace{-.8cm}\textbf{Requirements}: Cost function $f$, Manifold $\mathcal{M}$, Metric $g$, \\Initial point ${\bf{X}}_0 \in \mathcal{M}$, retraction $R$ defined from $T\mathcal{M}$ to $\mathcal{M}$,\\ scalars $\bar{\alpha}>0, \,  \, \beta,\sigma\in(0,1)$, \rm{and} tolerance $\tau>0$. \\
\hspace{-5.4cm}{\bf{for}}  $i=0,1,2,...$  {\bf{do}}\\
\textbf{Step 1}: \,\,\,\, Set $\xi$ as the negative direction of the gradient, \,\,\,\,\,\,\,\,\,\,\,\,\,\,\,\,\,\,\\
\,\,\,\,\,\,\,\,\, $\xi_i:=-{\rm{grad}}f({\bf{X}}_i)$\\
\hspace{-3.8cm}\textbf{Step 2}:\,\,\,\, Evaluate convergence,\\
\,\,\,\,\,\,\,\,\, \textbf{if} $\left\|\xi_i\right\| < \tau$, \textbf{then break}\\
\hspace{-3.0cm}\textbf{Step 3}:\,\,\,\, Find the smallest $m$ satisfying\\
\,\,\,\,\,\,\,\, $f({\bf{X}}_i)-f(R_{{\bf{X}}_i}(\bar{\alpha}\beta^m\xi_i))\geq \sigma\bar{\alpha}\beta^m g_{{\bf{X}}_i}(\xi_i,\xi_i)$\\
\hspace{-3.2cm}\textbf{Step 4}:\,\,\,\,\,\,\,\,\, Find the modified point as \\
\,\,\,\,\,\,\,\,\,\, ${\bf{X}}_{i+1}:=R_{{\bf{X}}_i}(\bar{\alpha}\beta^m\xi_i))$\\
\hline
\end{tabular}
\end{center}}

To elaborate on Algorithm 1, in Step 1,  we calculate the gradient of the cost function over our proposed manifold. Step 2 evaluates the convergence of the optimization problem and Step 3 is addressing the Armijo backtracking procedure to find a reasonable step size \cite{Absil_Book}. Step 4 retracts the updated point which lies on the tangent space to the manifold.

In Proposition \ref{Proposition_Ours_1}, we presented a new manifold which is an embedded submanifold of the manifold of constant rank matrices. Now, we present the required geometrical properties of this submanifold to be able to apply Algorithm 1 for  solving  (\ref{Alternating_Minimization_Surrogate_a}).

\begin{proposition}\label{gradient_over_ours}
Let ${\bf{X}} \in \mathcal{M}^{(r)({\bf{C}})}$, ${\bf{X}}={\bf{U}}_{m\times r}{\bf{\Sigma}}_{r\times r}{\bf{V}}_{n\times r}^T$ and $h$ be a linear operator. Then, the gradient of the function $f({\bf{X}})=\left\|h({\bf{X}})\right\|_F^2/2$ over the Riemannian manifold $\mathcal{M}^{(r)({\bf{C}})}$, defined in Proposition \ref{Proposition_Ours_1}, is
\begin{equation} \label{gradient_over_ours_Equ}
\begin{array}{l}
{{\rm{grad}}f({\bf{X}})} =  {\rm{P}}_{T_{\bf{X}}\mathcal{M}^{(r)({\bf{C}})}} (h({\bf{X}})),
\end{array}
\end{equation}
where ${\rm{P}}_{T_{\bf{X}}\mathcal{M}^{(r)({\bf{C}})}}$ is the orthogonal projection onto the tangent space ${T_{\bf{X}}\mathcal{M}^{(r)({\bf{C}})}}$ defined as
\begin{equation} \label{Orthogonal_Projection_Our_Manifold}
\begin{array}{l}
 {\rm{P}}_{T_{\bf{X}}\mathcal{M}^{(r)({\bf{C}})}}:{\mathbb{R}^{m\times n}}\to {T_{\bf{X}}\mathcal{M}^{(r)({\bf{C}})}}\\\,\,\,\,\,\,\,\,\,\,\,\,\,\,\,\,\,\,\,\,\,\,\,\,\,\,\,\,\,\,\,\,\,:{\bf{Z}}\,\,\,\,\,\,\,\,\,\,\,\to({\rm{P}}_{\bf{U}}{\bf{Z}}{\rm{P}}_{\bf{V}}+({\bf{I}}_m-{\rm{P}}_{\bf{U}}){\bf{Z}}{\rm{P}}_{\bf{V}}\\\,\,\,\,\,\,\,\,\,\,\,\,\,\,\,\,\,\,\,\,\,\,\,\,\,\,\,\,\,\,\,\,\,\,\,\,\,\,\,\,\,\,\,\,\,\,\,\,\,\,\,\,\,\,\,\,\,\,\,\,+{\rm{P}}_{\bf{U}}{\bf{Z}}({\bf{I}}_n-{\rm{P}}_{\bf{V}})){\rm{P}}_{\bf{W}}\\
\end{array}
\end{equation}
for the orthogonal projections ${\rm{P}}_{\bf{V}}={\bf{V}}{\bf{V}}^T$, ${\rm{P}}_{\bf{U}}={\bf{U}}{\bf{U}}^T$, and ${\rm{P}}_{\bf{W}}={\bf{W}}{\bf{W}}^T$. Moreover, ${\bf{W}}$ is a matrix whose columns are orthonormal vectors spanning the null space of ${\bf{(C-I)}}^T$ and ${\bf{I}}_m$ is the identity matrix of size $m\times m$.
\end{proposition}
\begin {proof}
Consider the Proposition \ref{gradient_submanifold} and note that the manifold $\mathcal{M}^{(r)({\bf{C}})}$ is an embedded submanifold of the manifold $\mathcal{M}^{(r)}$ and consequently an embedded submanifold of the Euclidean space $\mathbb{R}^{m\times n}$. Also, considering the linearity of the operator $h$, it is easy to verify that the gradient of the function $f({\bf{X}})=\left\|h({\bf{X}})\right\|_F^2/2$ over the Euclidean space $\mathbb{R}^{m\times n}$ is $h({\bf{X}})$. Now, we require to find the orthogonal projection to the tangent space ${T_{\bf{X}}\mathcal{M}^{(r)({\bf{C}})}}$, which is performed in two steps. In the first step, we need the orthogonal projection operator onto the tangent space ${T_{\bf{X}}\mathcal{M}^{(r)}}$. By considering  (\ref{tangent_space_CRM}), this operator can be written as ${\rm{P}}_{\bf{U}}{\bf{Z}}{\rm{P}}_{\bf{V}}+({\bf{I}}_m-{\rm{P}}_{\bf{U}}){\bf{Z}}{\rm{P}}_{\bf{V}}+{\rm{P}}_{\bf{U}}{\bf{Z}}({\bf{I}}_n-{\rm{P}}_{\bf{V}})$. In the second step, we should restrict the tangent vectors of the tangent space ${T_{\bf{X}}\mathcal{M}^{(r)}}$ to the tangent space ${T_{\bf{X}}\mathcal{M}^{(r)({\bf{C}})}}$. Because of the property ${\bf{XC}}={\bf{X}}$,  it is easy to show that a tangent vector $\xi\in {T_{\bf{X}}\mathcal{M}^{(r)({\bf{C}})}}$ should satisfy the property $\xi({\bf{C-I}})={\bf{0}}$. Therefore, in this step, we apply the orthogonal projection operator ${\rm{P}}_{\bf{W}}$ over any tangent vector of ${T_{\bf{X}}\mathcal{M}^{(r)}}$ to construct a tangent vector belonging to $ {T_{\bf{X}}\mathcal{M}^{(r)({\bf{C}})}}$.
\end {proof}
Now, consider the retraction of ${\bf{X}}'={\bf{X}}+\xi$ to the manifold $\mathcal{M}^{(r)({\bf{C}})}$, where ${\bf{X}}\in\mathcal{M}^{(r)({\bf{C}})}$ and  ${\xi}\in T_{\bf{X}}\mathcal{M}^{(r)({\bf{C}})}$. We allege that the result of the metric projection in Definition \ref{retraction} on the manifold $\mathcal{M}^{(r)({\bf{C}})}$ is easily obtained by the same procedure as that of $\mathcal{M}^{(r)}$, $i.e.$, taking SVD and keeping the $r$ largest singular values. Indeed, this rank-$r$ approximation not only belongs to the manifold of rank-$r$ matrices but also satisfies the self expressive property with respect to the expression matrix ${\bf{C}}$, and hence, belongs to $\mathcal{M}^{(r)({\bf{C}})}$. Let us verify satisfying the self expressive property by the rank-$r$ approximation of ${\bf{X}}'$. We know that  ${\bf{X}}$ belongs to the null space of the row space of ${\bf{C-I}}$, $i.e.$, ${\bf{X}}({\bf{C-I}})={\bf{0}}$. Therefore, by construction of $\xi$,  ${\bf{X}}'$ belongs to this null space as well. Now, when the rank-$r$ approximation, if exists, is a projection onto the null space of the row space of ${\bf{C-I}}$, the result does satisfy the self expressive property with respect to the expression matrix  ${\bf{C}}$. Furthermore, the existence of the rank-$r$ approximation dependes on the rank of ${\bf{X}}'$ which should be greater than or equal to $r$. Moreover, to have a unique result, we need in nonincreasing ordered singular values the $r^{th}$ singular value to be strictly greater than the $(r+1)^{th}$ one. As mentioned in Section \ref{sec:Preliminary}, these are why the retraction cannot be defined all over the tangent bundle and should be considered locally. The convergence of Algorithm 1 is given in \cite{Absil_Book}.
\section {Simulation Results} \label{Simulation}
 For sake of simulations, in the following we consider two concepts about noise. One is named the model noise defined by additive Gaussian noise which models side information uncertainties and the other one is the measurement noise. Also, note that we define the measurement SNR to address the ratio of samples to the measurement noise and the model SNR to address the ratio of the matrix of side information to the model noise. We consider the measurement noise in our simulations using the impulsive noise with Gaussian mixture model (GMM). To validate the performance of the proposed optimization problem  in (\ref{Alternating_Minimization_Surrogate}), we consider three different scenarios for matrix completion using synthetic data for both  noisy and noiseless measurement cases. 
To realize impulsive noise  Gaussian Mixture Model  (GMM), we added GMM noise to each observed element with probability 0.2. Specifically, consider the GMM defined as:
\\ $P = {\alpha _1}\mathcal{N}\left( {{\mu _1},\sigma _1^2} \right) + {\alpha _2}\mathcal{N}\left( {{\mu _2},\sigma _2^2} \right)$,\\
with  ${\mu _1} = 0.1$, ${\mu _2} = -0.2$,  $\alpha_1=0.3$, and $\alpha_2=0.7$. The variances $\sigma_1^2$  and $\sigma_2^2$  are set for the Measurement SNR of 8 dB. Please note that based on our simulations, for fixed   $\sigma_1^2$  and $\sigma_2^2$, as observation probability increases the Measurement SNR decreases. Accordingly, to fix the Measurement SNR at 8 dB for different observation probabilities, we decrease   $\sigma_1^2$  and $\sigma_2^2$   when observation probability increases.

In all cases, it is assumed that ${\bf{M}}$ lies in a union of subspaces. To form ${\bf{B}}$, we first generate $S$ random $d$-dimensional bases of $\mathbb{R}^{r}$ and then  generate $N_s$ vectors, $s=1,...,S$, corresponding to each basis and concatenate them where $\sum\nolimits_{s = 1}^S {{N_s}}  = n$.
 Then, we generate ${\bf{B}}'$  by adding zero mean white Gaussian noise to ${\bf{B}}$. Similarly, for the side information ${\bf{A}}$, an approximation of ${\bf{A}}'$ is considered. The entries of the sampling operator ${{\rm{P}}_\Omega }$ are drawn from a binary distribution where $1$ occurs with probability of $p$. In addition, for comparison purposes, the normalized  mean square error (NMSE) and residual normalized MSE (RNMSE) criteria  are defined as
\begin{equation} \label{NMSE}
\begin{array}{l}
{\rm{NMSE}} = \frac{{\left\| {{\bf{M}} - \widehat {\bf{M}}} \right\|_F^2}}{{\left\| {\bf{M}} \right\|_F^2}}\\ \\
{\rm{RNMSE}} = \frac{{\left\| {{{\rm{P}}_\Omega }({\bf{M}} - \widehat {\bf{M}})} \right\|_F^2}}{{\left\| {{{\rm{P}}_\Omega }({\bf{M}})} \right\|_F^2}},
\end{array}
\end{equation}
where $\widehat{\bf{M}}$ shows the completed matrix of the partially observed matrix ${\bf{M}}$.

In the first  noiseless measurement scenario, the  NMSEs and RNMSEs are shown  in Figs. \ref{NMSE_pd} and  \ref{RNMSE_pd} versus $p=0.25$ to $0.9$.
To generate    $\bf{M}$, the parameters $m, n, r, S$ and $ d$ are set to $20, 60, 12, 3$ and $4$, respectively, and all $N_s$ are  equal to $20$. In this case, ${\bf{B}}'$ and  ${\bf{A}}'$ are with the SNR of $20$ and $10$ dB, respectively.
For comparison purposes, we have considered the methods in  \cite{Chiang_Hsieh_Dhillon} and \cite{Elhamifar_Mat_Completion} developed based on side information as discussed in Section \ref{sec:Introduction} and also the methods in  \cite{Keshavan_Few_Entries} and \cite{Bart} which have only contemplated low rankness of the matrix as the side information. As seen, our proposed method outperforms all the other methods for $p>0.3$. The methods of \cite{Chiang_Hsieh_Dhillon} and \cite{Elhamifar_Mat_Completion} fail to truly complete the matrix, because they lean on noiseless matrices $\bf{A}$ and $\bf{B}$ as side information. Also,  due to not using the side information in \cite{Keshavan_Few_Entries} and \cite{Bart}, the methods fail to accurately complete ${\bf{M}}$.
 It is worthy to note that the method of \cite{Bart}, which works on the manifold of constant rank matrices, cannot truly complete the matrix, while our proposed method which works on an embedded submanifold of the manifold of constant rank matrices can perform more successfully.
  Also, from the figures, one can infer that even for high observation probabilities, a better RNMSE does not necessarily correspond to a better NMSE.
For example, even thoguh the method of \cite{Bart} offers a better RNMSE performance than that of our proposed method, it has a worse NMSE. Note that the NMSE is more reliable in revealing the performance of the matrix completion methods.
\begin{figure}[!ht]
\centering
\begin{minipage}[b]{0.4\textwidth}
    \includegraphics[width=60mm,scale=0.5]{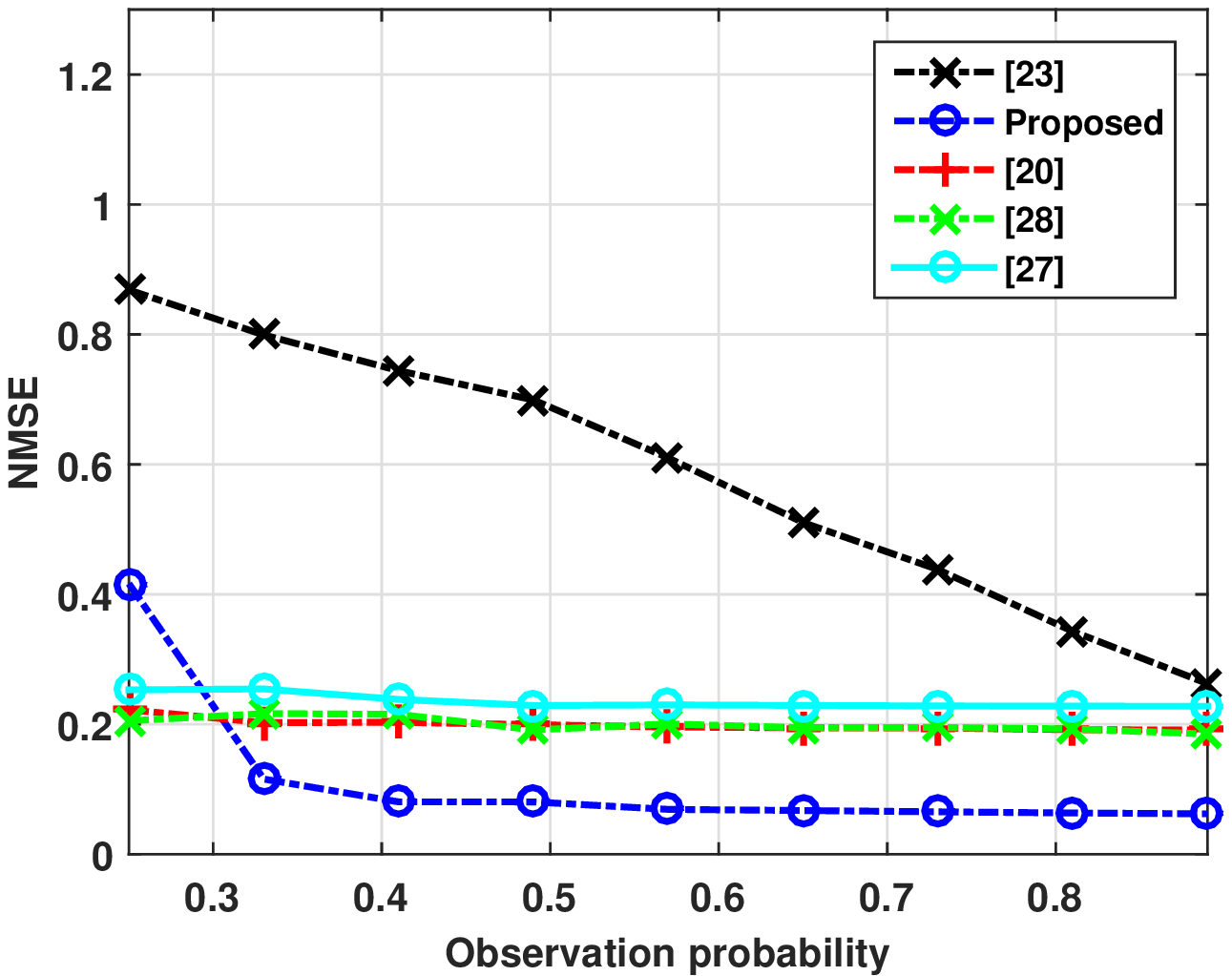}
  \caption{\label{NMSE_pd} Comparison of NMSEs for matrix completion versus the observation probability, $p$, for $r=12$ and Model SNR values $20$ and $10$ dB corresponding to $\bf{B}'$ and $\bf{A}'$, respectively.}
\end{minipage}
\hfill
  \centering
\begin{minipage}[b]{0.4\textwidth}
    \includegraphics[width=60mm,scale=0.5]{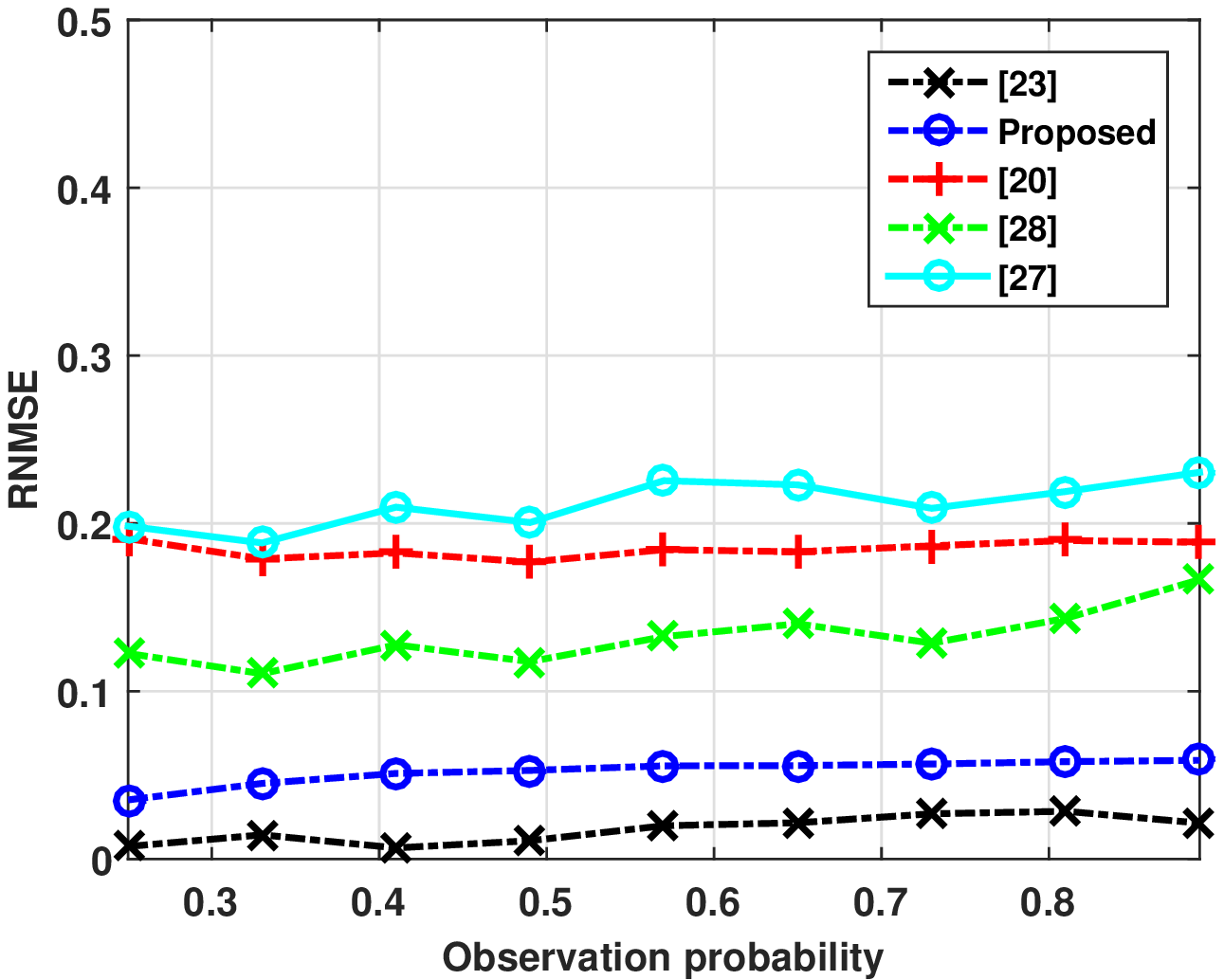}
  \caption{\label{RNMSE_pd} Comparison of RNMSEs for matrix completion versus the observation probability, $p$, for $r=12$ and Model SNR values $20$ and $10$ dB corresponding to $\bf{B}'$ and $\bf{A}'$, respectively.}
  \centering
\end{minipage}
\end{figure}\\
In the second  noiseless measurement scenario, we examine the performance of the aforementioned MC methods for a fixed value of $p=0.4$. Also, the SNR changes for ${\bf{B}'}$  from $5$ to $100$ dB and is set for ${\bf{A}'}$ to $10$ dB. The rest of parameters are similar to the first scenario except $r, S$, and $d$ which are equal to $15, 3$, and $5$, respectively. In the results shown in Figs.  \ref{NMSE_SNR} and  \ref{RNMSE_SNR}, \cite{Keshavan_Few_Entries} and \cite{Bart}  have generated constant values for different SNRs due to not using side information.
In contrast, the proposed algorithm yields lower NMSEs than those of the other methods for SNRs>$15$ dB. In addition, for SNRs>$50$ dB the RNMSE of our method is almost zero. Once again, the methods of \cite{Chiang_Hsieh_Dhillon} and \cite{Elhamifar_Mat_Completion} ,  offer no appealing results owing to leaning on noiseless matrices $\bf{A}$ and $\bf{B}$.
\begin{figure}[!ht]
\centering
\begin{minipage}[b]{0.4\textwidth}
    \includegraphics[width=60mm,scale=0.5]{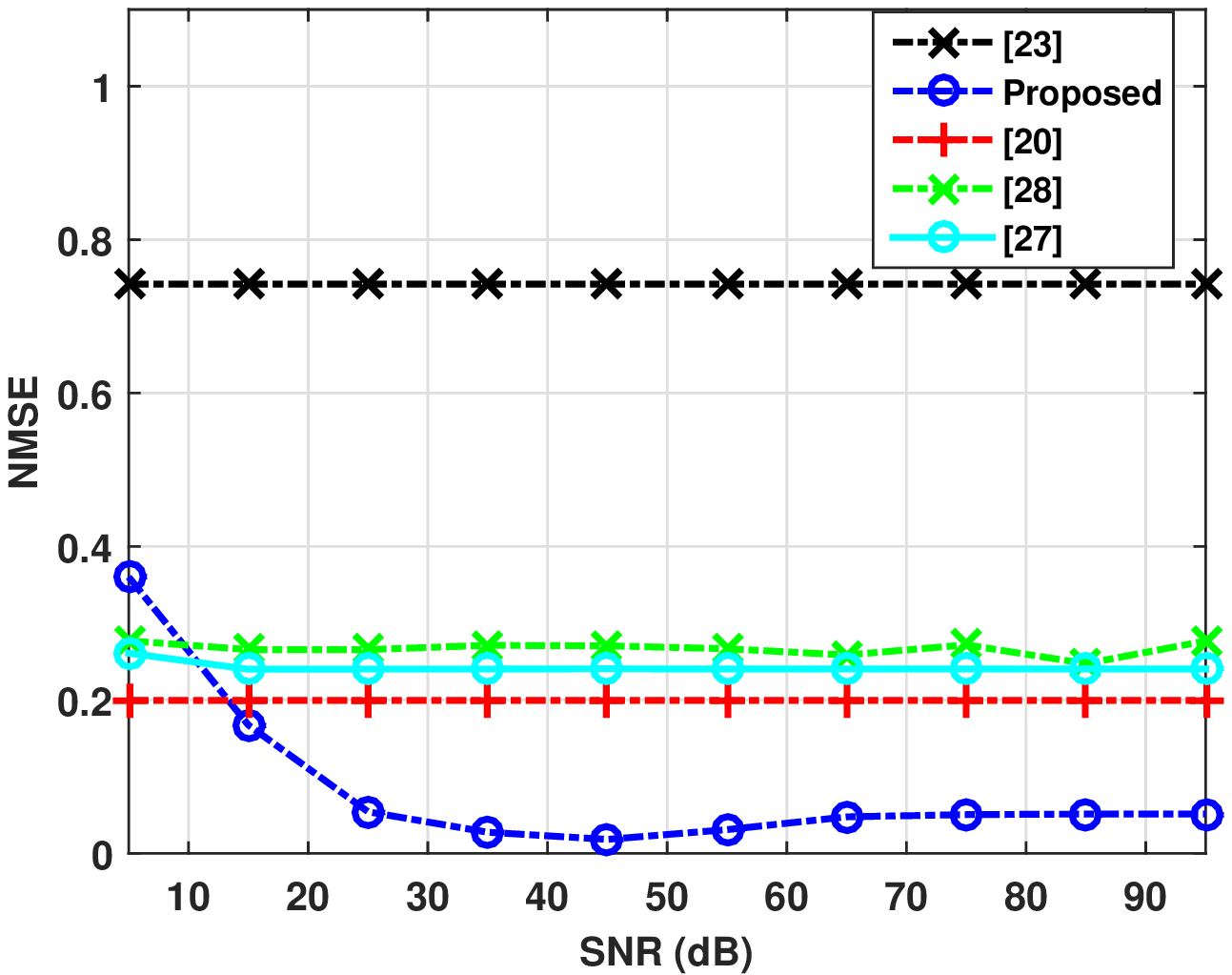}
  \caption{\label{NMSE_SNR} Comparison of NMSEs for matrix completion versus Model SNR  of ${\bf{B}}'$, for $r=15$, observation probability $0.4$ and SNR value of $10$ dB for ${\bf{A}}'$.}
\end{minipage}
\hfill
  \centering
\begin{minipage}[b]{0.4\textwidth}
    \includegraphics[width=60mm,scale=0.5]{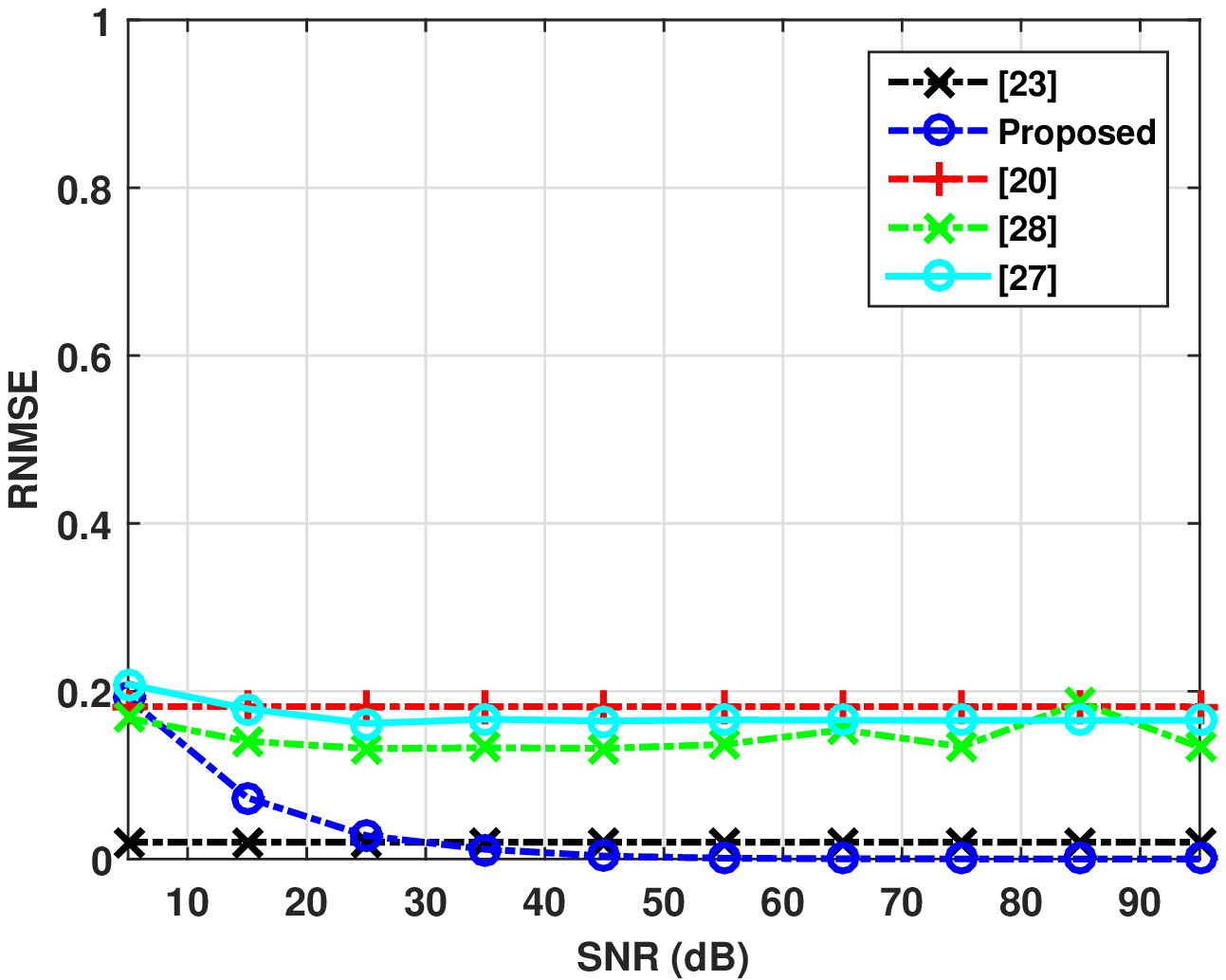}
  \caption{\label{RNMSE_SNR} Comparison of RNMSEs for matrix completion versus Model SNR of ${\bf{B}}'$, for $r=15$, observation probability $0.4$ and SNR value of $10$ dB for ${\bf{A}}'$.}
  \centering
\end{minipage}
\end{figure}\\
In the third  noiseless measurement scenario, the impact of the rank of a matrix on the mentioned MC algorithms is evaluated. In  this case, the parameters $m, n, p, d$ and SNR of ${\bf{A}'}$ and ${\bf{B}'}$ are set to $20, 64, 0.3, 2$, $5$ dB and $15$ dB, respectively. To produce matrices with different ranks, the values of $S$ is taken between $2$ to $10$. Considering the values of $d$ and $S$, it is deduced that the rank of the generated matrices would be $4, 6, 8, ...,20$. Also, note that  the equality $\sum\nolimits_{s = 1}^S {{N_s}}  = 64$ holds, $e.g.,$ for $S = 6$, the values of $N_s$ are  $9, 9, 9, 10, 13 \, {\rm{and}} \,14$. To have smooth curves, $50$ different runs have been averaged. As seen in Figs. \ref{NMSE_Rank} and \ref{RNMSE_Rank},  our proposed method like previous results outweighs the other MC methods. Once again, the results of \cite{Bart} resemble ours in RNMSEs, but are completely different in NMSEs;  which is due to not utilizing the  existing side information in  \cite{Bart}. Also, our method performs well for high rank matrices even for the full rank case of $r=20$. Note that while the methods of \cite{Chiang_Hsieh_Dhillon} and \cite{Elhamifar_Mat_Completion}  only perform better for higher rank matrices, our method can more widely yield promising results for all the ranks. \\
\begin{figure}[!ht]
\centering
\begin{minipage}[b]{0.4\textwidth}
    \includegraphics[width=60mm,scale=0.5]{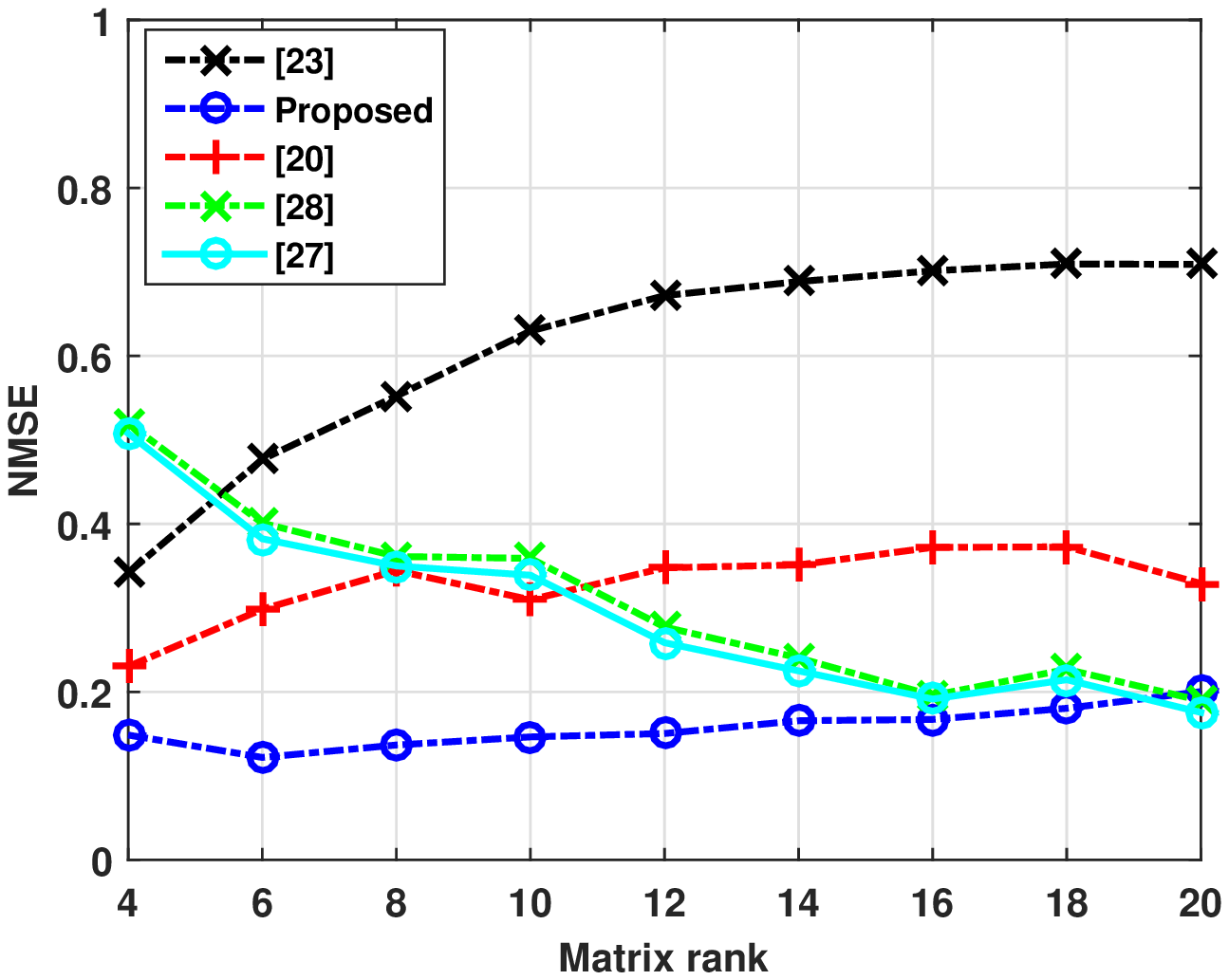}
  \caption{\label{NMSE_Rank}Comparison of NMSEs for matrix completion versus matrix rank, for observation probability $0.3$ and Model SNR values $15$ and $5$ dB corresponding to ${\bf{B}}'$ and ${\bf{A}}'$, respectively.}
\end{minipage}
\hfill
  \centering
\begin{minipage}[b]{0.4\textwidth}
    \includegraphics[width=60mm,scale=0.5]{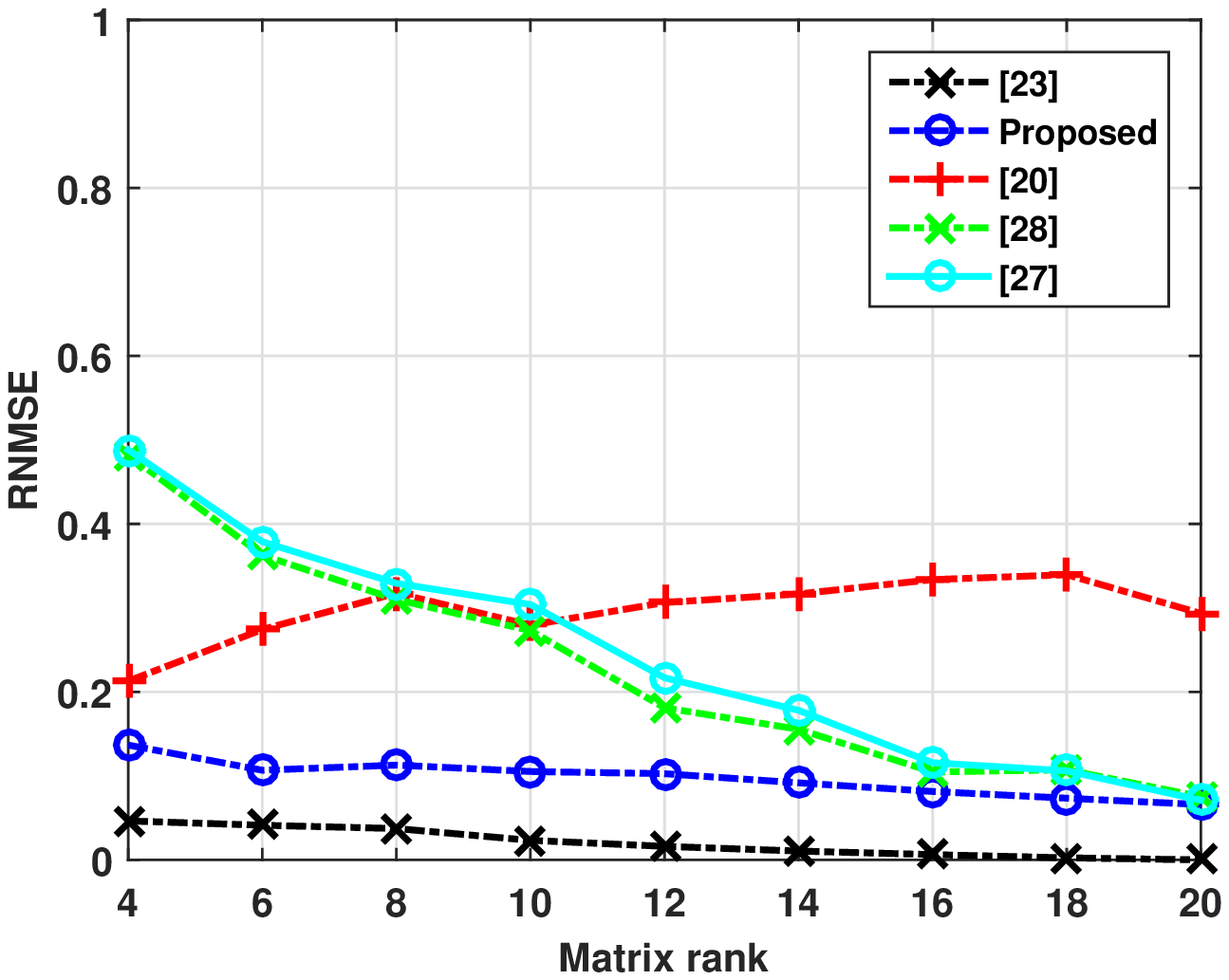}
  \caption{\label{RNMSE_Rank} Comparison of RNMSEs for matrix completion versus matrix rank, for observation probability $0.3$ and Model SNR values $15$ and $5$ dB corresponding to ${\bf{B}}'$ and ${\bf{A}}'$, respectively.}
  \centering
\end{minipage}
\end{figure}

Now, we evaluate the performance of the proposed method in presence of Measurement Noise. Simulation results are shown for the following scenarios:\\
1-Completion error versus observation probability.\\
2-Completion error versus the Model SNR.\\
3-Completion error versus the matrix rank.\\
The settings of the above cases are similar to those for generating Figs. \ref{NMSE_pd}, \ref{NMSE_SNR} and \ref{NMSE_Rank}, respectively.

For the first Noisy Measurement case, Figs. \ref{NMSE_pd_Impulsive_1_2} and \ref{RNMSE_pd_Impulsive_1_2} depict NMSE and RNMSE versus observation probability.
\begin{figure}[!ht]
\centering
\begin{minipage}[b]{0.4\textwidth}
    \includegraphics[width=60mm,scale=0.5]{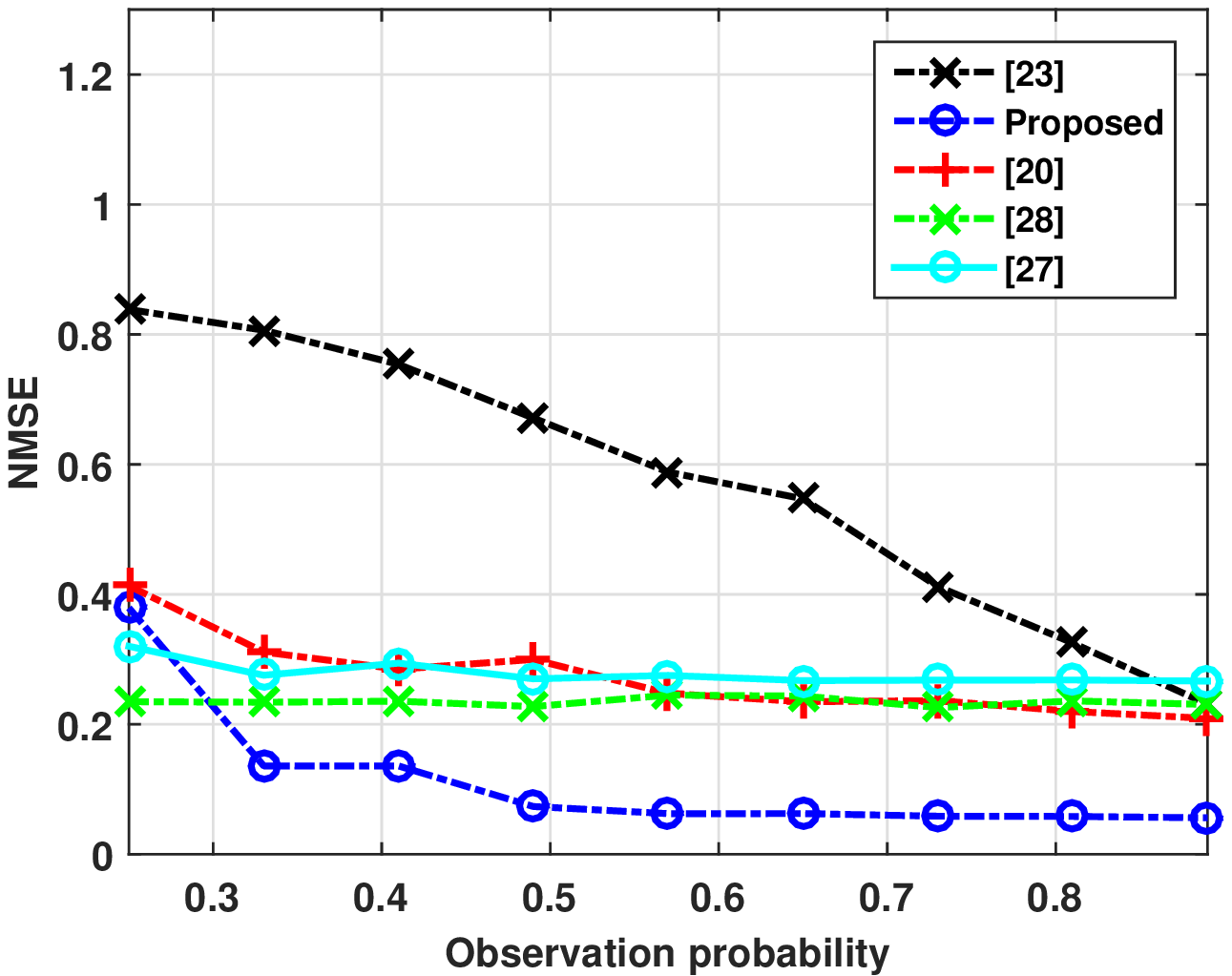}
  \caption{\label{NMSE_pd_Impulsive_1_2} Comparison of NMSEs for matrix completion versus the observation probability, $p$, for $r=12$ and Model SNR values $20$ and $10$ dB corresponding to $\bf{B}'$ and $\bf{A}'$, respectively, and Measurement SNR of 8 dB for each observation probability value.}
\end{minipage}
\hfill
  \centering
\begin{minipage}[b]{0.4\textwidth}
    \includegraphics[width=60mm,scale=0.5]{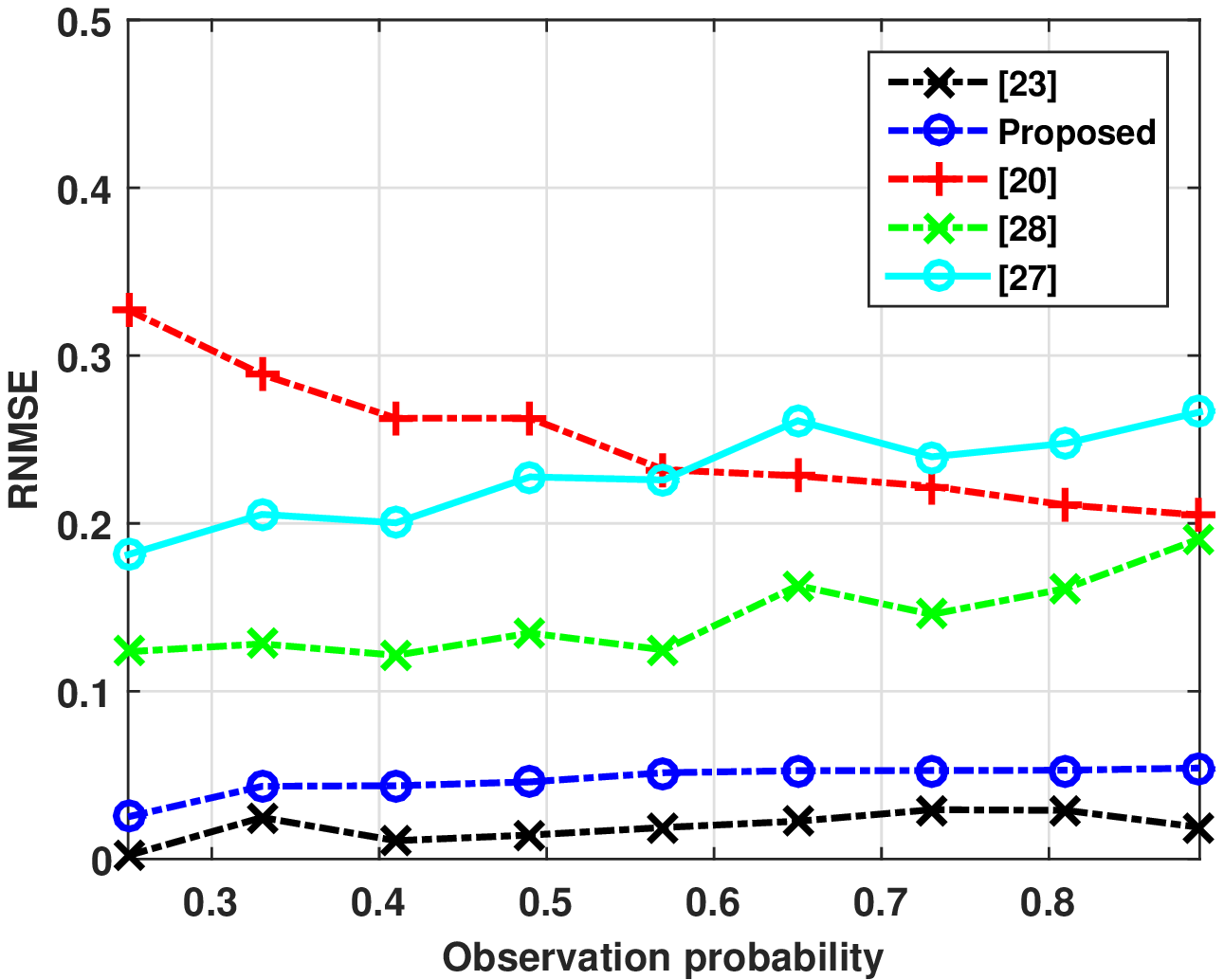}
  \caption{\label{RNMSE_pd_Impulsive_1_2} Comparison of RNMSEs for matrix completion versus the observation probability, $p$, for $r=12$ and Model SNR values $20$ and $10$ dB corresponding to $\bf{B}'$ and $\bf{A}'$, respectively, and Measurement SNR of 8 dB for each observation probability value.}
  \centering
\end{minipage}
\end{figure}\\
For the second Noisy Measurement  case, Figs. \ref{NMSE_SNR_Impulsive_3_4} and \ref{RNMSE_SNR_Impulsive_3_4} depict NMSE and RNMSE versus Model SNR.
\begin{figure}[!ht]
\centering
\begin{minipage}[b]{0.4\textwidth}
    \includegraphics[width=60mm,scale=0.5]{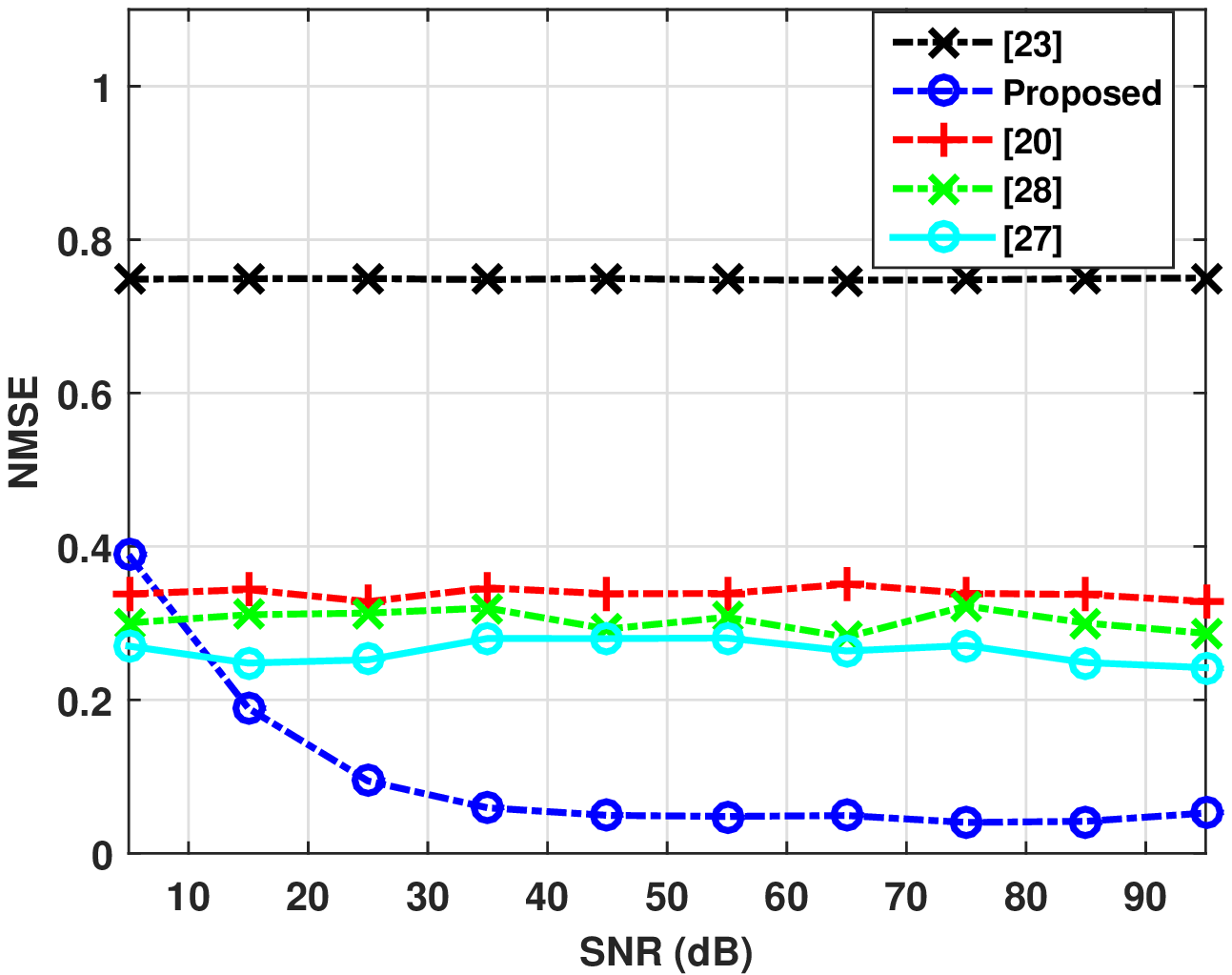}
  \caption{\label{NMSE_SNR_Impulsive_3_4} Comparison of NMSEs for matrix completion versus Model SNR  of ${\bf{B}}'$, for $r=15$, observation probability $0.4$, Model SNR value of $10$ dB for ${\bf{A}}'$, and Measurement SNR of 8 dB.}
\end{minipage}
\hfill
  \centering
\begin{minipage}[b]{0.4\textwidth}
    \includegraphics[width=60mm,scale=0.5]{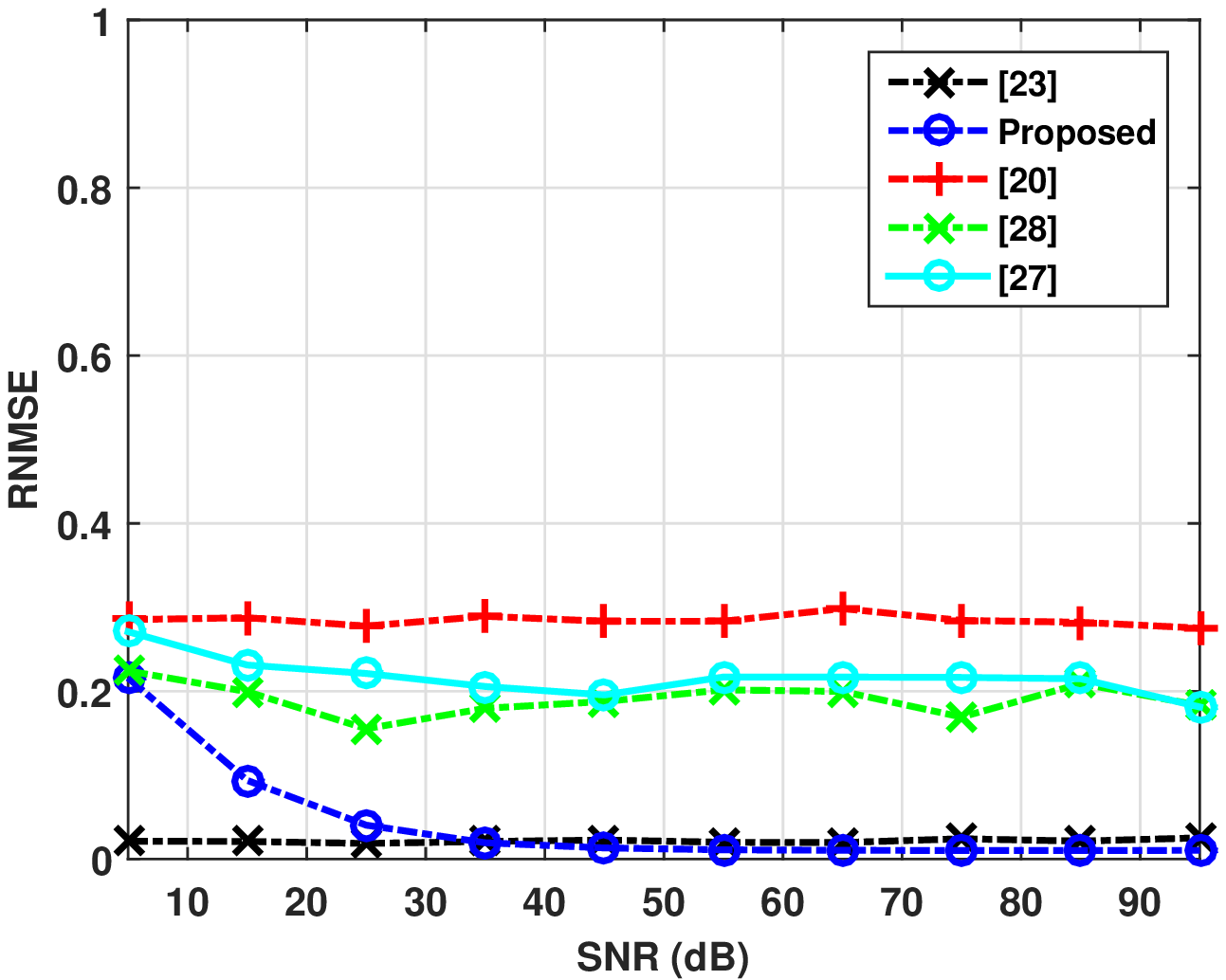}
  \caption{\label{RNMSE_SNR_Impulsive_3_4} Comparison of RNMSEs for matrix completion versus Model SNR  of ${\bf{B}}'$, for $r=15$, observation probability $0.4$, Model SNR value of $10$ dB for ${\bf{A}}'$, and Measurement SNR of 8 dB.}
  \centering
\end{minipage}
\end{figure}\\
In the third Noisy Measurement case, Figs. \ref{NMSE_Rank_Impulsive_5_6} and \ref{RNMSE_Rank_Impulsive_5_6} depict NMSE and RNMSE versus matrix rank.
\begin{figure}[!ht]
\centering
\begin{minipage}[b]{0.4\textwidth}
    \includegraphics[width=60mm,scale=0.5]{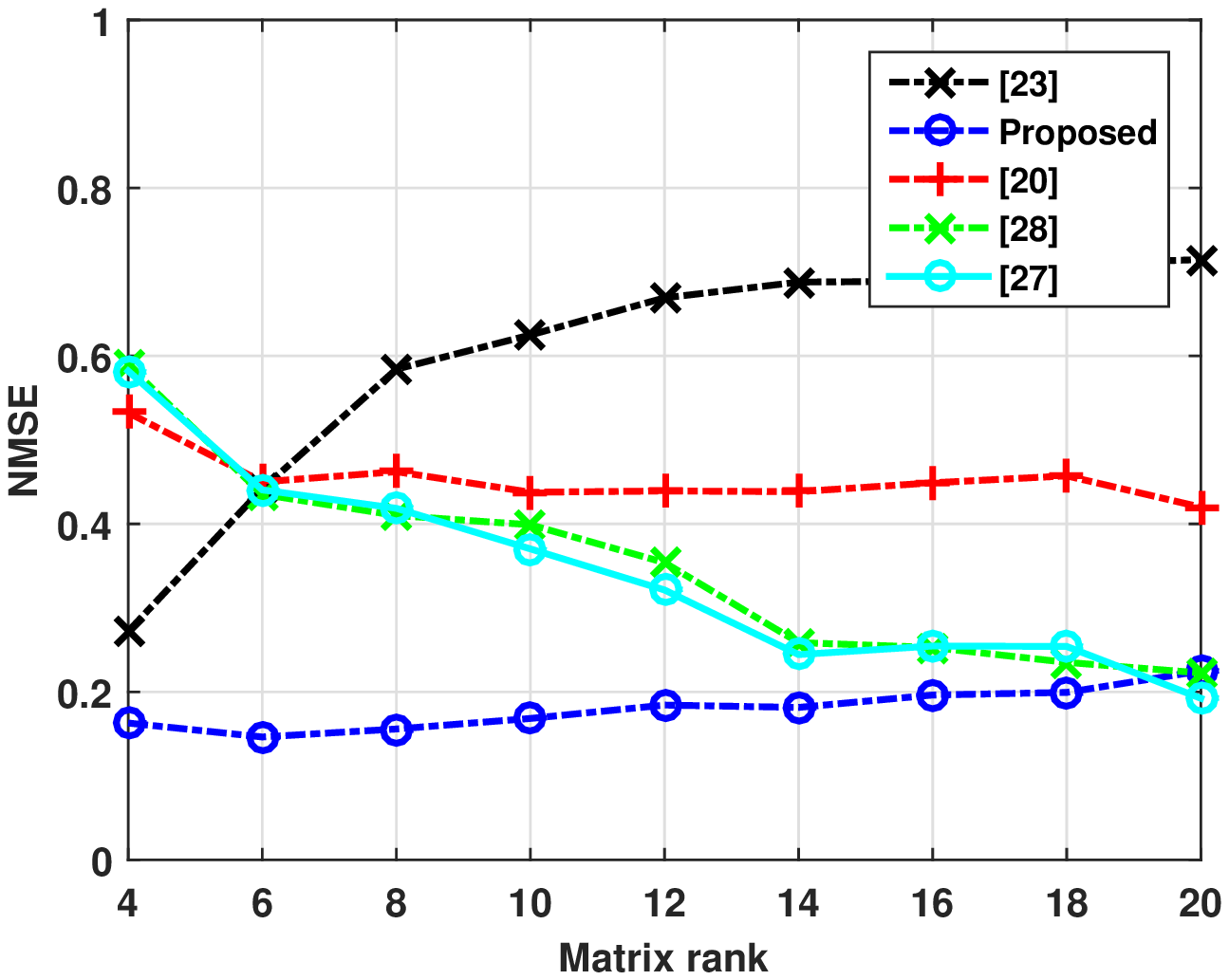}
  \caption{\label{NMSE_Rank_Impulsive_5_6} Comparison of NMSEs for matrix completion versus matrix rank, for observation probability $0.3$, Model SNR values $15$ and $5$ dB corresponding to ${\bf{B}}'$ and ${\bf{A}}'$, respectively, and Measurement SNR of 8 dB.}
\end{minipage}
\hfill
  \centering
\begin{minipage}[b]{0.4\textwidth}
    \includegraphics[width=60mm,scale=0.5]{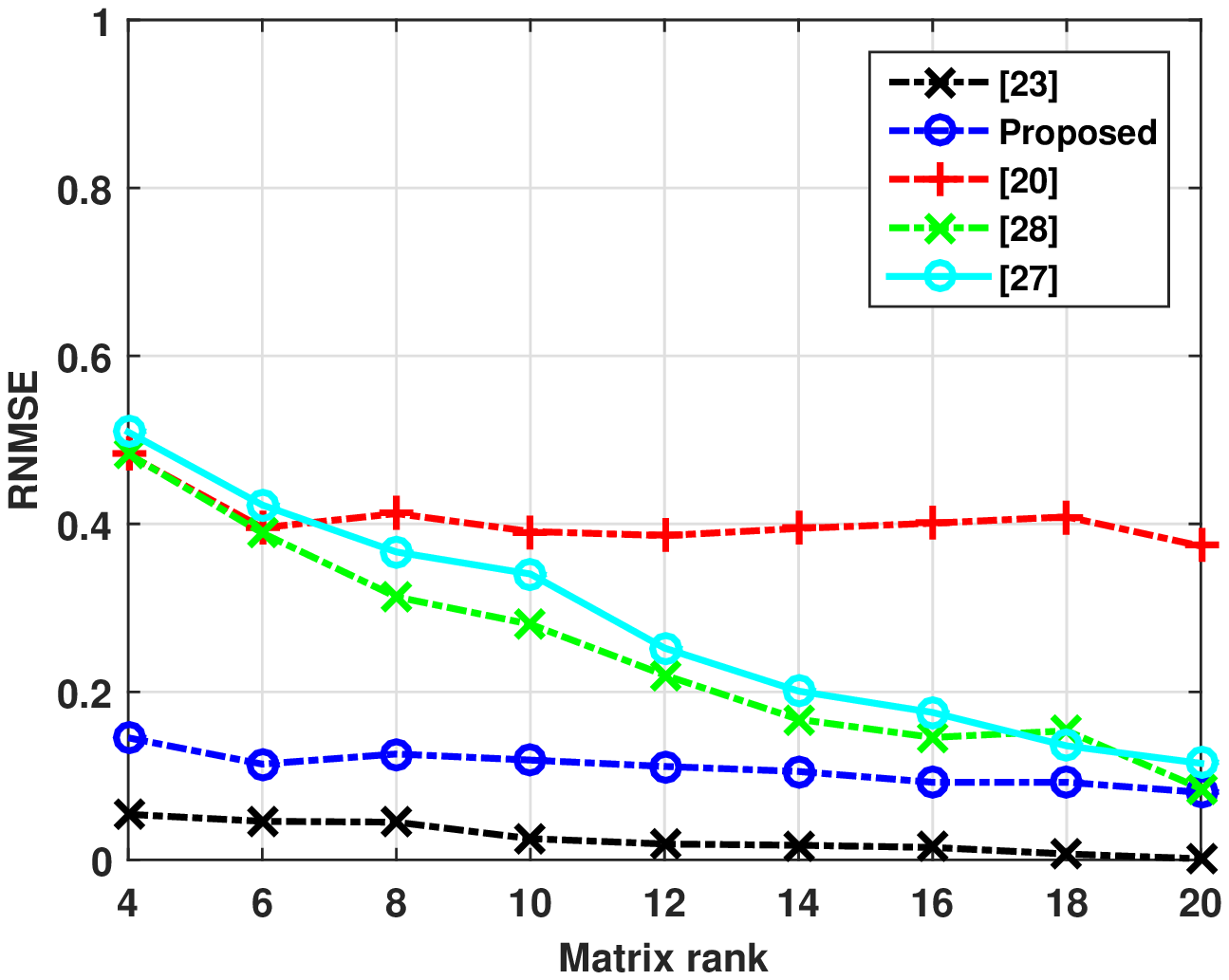}
  \caption{\label{RNMSE_Rank_Impulsive_5_6} Comparison of RNMSEs for matrix completion versus matrix rank, for observation probability $0.3$ and SNR values $15$ and $5$ dB corresponding to ${\bf{B}}'$ and ${\bf{A}}'$, respectively, and Measurement SNR of 8 dB.}
  \centering
\end{minipage}
\end{figure}

As seen, in general, all the methods are influenced by the impulsive noise. The method of \cite{Bart} shows most degradation compared to the other methods. Also, even though the method of \cite{Keshavan_Few_Entries} has experienced small performance degradation, it is the least reliable method. Our proposed method has been also influenced by impulsive noise, however it is yet more reliable than the other methods.

Next, we compare the runtime of different methods. Figure \ref{Runtime_pd} shows the runtime versus observation probability. As seen, our method runtime is comparable to the other methods. Note that although \cite{Keshavan_Few_Entries} and \cite{Bart} achieve lower runtime due to not using the side information, they result in less accurate estimates compared to the proposed method. Also, \cite{Chiang_Hsieh_Dhillon} and \cite{Elhamifar_Mat_Completion} require more runtime as observation probability increases, which is to our mind because of relying on noisy side information.

\begin{figure}[!ht]
  \centering
\begin{minipage}[b]{0.4\textwidth}
    \includegraphics[width=60mm,scale=0.5]{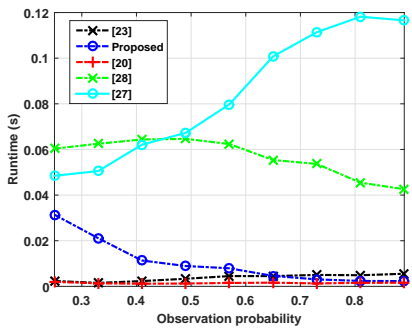}
  \caption{\label{Runtime_pd} Runtime per iteration for matrix completion versus the observation probability, $p$, for $r=12$ and SNR values $20$ and $10$ dB corresponding to $\bf{B}'$ and $\bf{A}'$, respectively.}
\end{minipage}
\hfill
\end{figure}

\section {Conclusion} \label{Conclusion}
We considered the problem of matrix completion with the assumption that  the column (or row) space of the matrix lies in a union of low dimensional subspaces which is called the self-expressive property. A non-convex MC problem was defined by addressing both constant rank constraint and the self-expressive property. Then, by developing an alternating minimization approach, our problem was split into two non-convex problems where the first one was relaxed to a convex problem. For the second problem, we proved that the self expressive property can be employed to construct ${\it{embedded}}$ submanifold of the manifold of constant rank matrices, and proposed a manifold optimization problem.
Simulation results confirmed that our proposed method outperforms the existing MC methods. In spite of most of related  methods which incorporate accurate side information, our method only requires  perturbed side information.



\renewcommand{\theequation}{A-\arabic{equation}}
  \setcounter{equation}{0}  

\end{document}